\documentclass[11pt, a4paper]{article}

\usepackage{typearea}
\typearea{12}

\usepackage{amsmath}
\usepackage{amsthm}
\usepackage{amssymb}
\usepackage{color}
\usepackage{url}
\usepackage{graphicx}
\usepackage{float}
\usepackage{comment}
\usepackage{ifthen}

\newcommand{\dd}{\mathrm{d}}
\newcommand{\ii}{\mathrm{i}}

\theoremstyle{definition}
\newtheorem{thm}{Theorem}
\newtheorem{prop}[thm]{Proposition}
\newtheorem{lem}[thm]{Lemma}
\newtheorem{cor}[thm]{Corollary}

\newtheorem{rem}{Remark}
\newtheorem{assump}{Assumption}

\allowdisplaybreaks[4]

\numberwithin{equation}{section}
\numberwithin{thm}{section}
\numberwithin{rem}{section}

\newcounter{journal}
\title{
Design of accurate formulas for approximating functions 
in weighted Hardy spaces by discrete energy minimization}
\author{
Ken'ichiro Tanaka%
\footnote{
Department of Mathematical Informatics,
Graduate School of Information Science and Technology,
University of Tokyo. 
7-3-1 Hongo, Bunkyo-ku, Tokyo, 113-8656, Japan. 
\texttt{kenichiro@mist.i.u-tokyo.ac.jp}}, 
Masaaki Sugihara%
\footnote{
Department of Physics and Mathematics, 
College of Science and Engineering, Aoyama Gakuin University. 
5-10-1, Fuchinobe, Chuo-ku, Sagamihara-shi, Kanagawa 252-5258, Japan. }
}
\date{January 13, 2018}

\begin{document}

\maketitle

\begin{abstract}
We propose a simple and effective method for 
designing approximation formulas for weighted analytic functions.
We consider spaces of such functions according to weight functions 
expressing the decay properties of the functions. 
Then, we adopt the minimum worst error of the 
$n$-point approximation formulas in each space 
for characterizing the optimal sampling points for the approximation. 
In order to obtain approximately optimal sampling points, 
we consider minimization of a discrete energy related to the minimum worst error. 
Consequently, 
we obtain an approximation formula 
and its theoretical error estimate in each space. 
In addition, from some numerical experiments, 
we observe that the formula generated by the proposed method 
outperforms the corresponding formula derived with sinc approximation, 
which is near optimal in each space. 
\end{abstract}


\section{Introduction}

In 
\ifthenelse{\value{journal} = 0}{%
the recent paper \cite{bib:TaOkaSu_PotApprox_2017}, 
}{
\cite{bib:TaOkaSu_PotApprox_2017},
}
the authors proposed a method for designing interpolation formulas on $\mathbf{R}$
for approximating functions in spaces of weighted analytic functions. 
They considered the weighted Hardy space defined by
\begin{align}
\label{eq:def_weighted_Hardy_rev}
\boldsymbol{H}^{\infty}(\mathcal{D}_{d}, w)
:=
\left\{
f : \mathcal{D}_{d} \to \mathbf{C} 
\ \left| \
f \text{ is analytic on } \mathcal{D}_{d} \text{ and } 
\| f \| < \infty 
\right.
\right\}, 
\end{align}
where $d > 0$, $\mathcal{D}_{d} = \{ z \in \mathbf{C} \mid \, | \mathop{\mathrm{Im}} z | < d \}$, 
$w$ is a weight function with $w(z) \neq 0$ for any $z \in \mathcal{D}_{d}$, and 
\begin{align}
\label{eq:def_weighted_Hardy_norm_rev}
\| f \| := \sup_{z \in \mathcal{D}_{d}} \left| \frac{f(z)}{w(z)} \right|. 
\end{align}
In this study, we drastically simplify the method and obtain approximation formulas 
in the spaces $\boldsymbol{H}^{\infty}(\mathcal{D}_{d}, w)$, 
competitive with the formulas reported previously. 
Furthermore, we broaden the class of the weight functions $w$ to which the method is applicable. 

The space $\boldsymbol{H}^{\infty}(\mathcal{D}_{d}, w)$
is important as a space of transformed functions 
that often appear in transformation-based formulas
for function approximation in sinc numerical methods
\ifthenelse{\value{journal} = 0}{%
\cite{bib:Stenger1993, bib:Stenger2011, bib:Sugihara_NearOpt_2003, bib:TaSuMu_DE_Sinc_2009}. 
}{
\citep{bib:Stenger1993, bib:Stenger2011, bib:Sugihara_NearOpt_2003, bib:TaSuMu_DE_Sinc_2009}. 
}
These numerical methods are based on the sinc function $\mathop{\mathrm{sinc}}(x) = (\sin \pi x)/(\pi x)$ 
with a useful variable transformation $\psi$.  
The building block of the method is the sinc approximation given by
\begin{align}
f(x) \approx \sum_{k = -N_{-}}^{N_{+}} f(kh) \, \mathop{\mathrm{sinc}}(x/h - k), 
\label{eq:trans_sinc_approx}
\end{align}
where  $h > 0$. 
Usually, we consider approximation of an analytic function $g$ defined on a domain $D \subset \mathbf{C}$. 
Then, we employ a map $\psi: \mathcal{D}_{d} \to D$ as a variable transformation
and apply the sinc approximation in~\eqref{eq:trans_sinc_approx} to the transformed function $f(x) = g(\psi(x))$ for $x \in \mathbf{R}$. 
If the function $g$ has a decay property, 
the map $\psi$ achieves the decay of function $f$,  
which enables us to select $N_{-}$ and $N_{+}$ to be small in the sum in~\eqref{eq:trans_sinc_approx}. 
Owing to this simple principle, the sinc interpolation is useful for various numerical methods. 
Typical maps $\psi = \psi_{i}$ used as such transformations are%
\footnote{
\ifthenelse{\value{journal} = 0}{%
These maps are also used for numerical integration based on a variable transformation. 
}{
These maps are also used for numerical integration based on a variable transformation. 
The maps $\psi_{1}$ and $\psi_{2}$ are used in 
\cite{bib:Haber_tanh_1977, bib:Schwartz_NumInt_1969, bib:Stenger1993, bib:Stenger2011, bib:TakahasiMoriVT1973}
and \cite{bib:TakahasiMoriDE1974}, respectively.
}}:
\begin{align}
& \psi_{1}(x) = \tanh \left( \frac{x}{2} \right) \qquad 
\text{(TANH transformation%
\ifthenelse{\value{journal} = 0}{%
~\cite{bib:Haber_tanh_1977, bib:Schwartz_NumInt_1969, bib:Stenger1993, bib:Stenger2011, bib:TakahasiMoriVT1973}%
}{
})}, 
\label{eq:TANHtrans} \\
& \psi_{2}(x) = \tanh \left( \frac{\pi}{2} \sinh x \right) \qquad 
\text{(DE transformation%
\ifthenelse{\value{journal} = 0}{%
~\cite{bib:TakahasiMoriDE1974}%
}{
})},
\label{eq:DEtrans}
\end{align}
where ``DE'' denotes ``double exponential''.
Therefore, we consider a weight function $w$ to represent the decay property of $f$ given by $\psi$. 

\ifthenelse{\value{journal} = 0}{%
Sugihara~\cite{bib:Sugihara_NearOpt_2003}%
}{
\cite{bib:Sugihara_NearOpt_2003}%
}
showed that the sinc interpolation was a ``nearly optimal'' approximation in $\boldsymbol{H}^{\infty}(\mathcal{D}_{d}, w)$
for several weight functions $w$. 
He adopted minimum worst error 
$E_{n}^{\mathrm{min}}(\boldsymbol{H}^{\infty}(\mathcal{D}_{d}, w))$
of an $n$-point approximation formula in $\boldsymbol{H}^{\infty}(\mathcal{D}_{d}, w)$,
whose definition is given later in~\eqref{eq:def_E_min}, 
and showed that the error in the sinc interpolation for the functions in 
$\boldsymbol{H}^{\infty}(\mathcal{D}_{d}, w)$ was close to
$E_{n}^{\mathrm{min}}(\boldsymbol{H}^{\infty}(\mathcal{D}_{d}, w))$. 
However, an explicit optimal formula attaining $E_{n}^{\mathrm{min}}(\boldsymbol{H}^{\infty}(\mathcal{D}_{d}, w))$
is known only in a limited case 
\ifthenelse{\value{journal} = 0}{%
\cite{bib:TaOkaSu_Ganelius_arXiv2016}. 
}{
\citep{bib:TaOkaSu_Ganelius_arXiv2016}. 
}

\ifthenelse{\value{journal} = 0}{%
In the paper \cite{bib:TaOkaSu_PotApprox_2017}, 
}{
In \cite{bib:TaOkaSu_PotApprox_2017}, 
}
with a view to an optimal formula in $\boldsymbol{H}^{\infty}(\mathcal{D}_{d}, w)$ 
for a general weight function $w$, 
the authors started with the expression 
\begin{align}
E_{n}^{\mathrm{min}}(\boldsymbol{H}^{\infty}(\mathcal{D}_{d}, w))
& = 
\inf_{a_{j} \in \mathbf{R}} 
\left[
\sup_{x \in \mathbf{R}}
\left| 
w(x) 
 \prod_{j = 1}^{n} \tanh \left( \frac{\pi}{4d} (x-a_{j}) \right)
\right|
\right]  
\label{eq:CharMinErr_Intro}
\end{align}
given in~\cite{bib:Sugihara_NearOpt_2003}, 
and employed fundamental tools in potential theory%
\ifthenelse{\value{journal} = 0}{%
~\cite{bib:SaffTotik_LogPotExtField_1997}%
}{
~\citep{bib:SaffTotik_LogPotExtField_1997}%
}
to obtain an accurate approximation formula. 
Their method was based on approximating the equilibrium measure that minimized
the weighted Green energy by considering the integral equation 
corresponding to a Frostman type characterization of the equilibrium measure. 
The integral equation was slightly complex 
because it contained unknown parameters representing the support of the  equilibrium measure. 
For simplicity, 
the authors limited their study to weight functions that are even on $\mathbf{R}$
\ifthenelse{\value{journal} = 0}{%
\cite[Assumption 2.2]{bib:TaOkaSu_PotApprox_2017}. 
}{
\citep[Assumption 2.2]{bib:TaOkaSu_PotApprox_2017}. 
}
Subsequently, 
they used some heuristic techniques to derive an approximate density function for each equilibrium measure, 
and obtained a sequence of sampling points for interpolation by discretizing the density function. 

In this paper, 
we propose a simplified method for obtaining sampling points 
for approximating functions in $\boldsymbol{H}^{\infty}(\mathcal{D}_{d}, w)$.
This method is based on discrete energy minimization,
which determines the sampling points directly. 
It can be considered as a type of method that generates a good point configuration
by minimizing a certain functional, such as the Riesz energy%
\ifthenelse{\value{journal} = 0}{%
~\cite{bib:BrauchartGrabner_Survey2015}. 
}{
~\citep{bib:BrauchartGrabner_Survey2015}. 
}
Essentially, the proposed method is a discrete analogue of the minimization of the weighted Green energy. 
In general, 
discrete energy minimization is 
not easily tractable computationally
because it is not always a convex optimization problem. 
Then, we assume the strict log-concavity%
\footnote{Simple log-concavity is assumed for $w$ in 
\ifthenelse{\value{journal} = 0}{%
\cite[Assumption 2.3]{bib:TaOkaSu_PotApprox_2017}. 
}{
\citet[Assumption 2.3]{bib:TaOkaSu_PotApprox_2017}. 
}} 
of a weight function $w$ on $\mathbf{R}$.
On this assumption, 
the minimization problem becomes convex and 
we can show that it has a unique optimal solution and that it 
is characterized by a stationary condition. 
Moreover, we can compute it by a standard technique of convex optimization. 
In addition, we can deal with weight functions $w$ that are not even on $\mathbf{R}$, 
i.e., we can deal with a wider class of the spaces $\boldsymbol{H}^{\infty}(\mathcal{D}_{d}, w)$ 
than the previous study 
\ifthenelse{\value{journal} = 0}{%
\cite{bib:TaOkaSu_PotApprox_2017}.
}{
\citep{bib:TaOkaSu_PotApprox_2017}.
}

The rest of this paper is organized as follows.
Section~\ref{sec:prelim}
presents mathematical preliminaries, including some fundamental tools in potential theory.
In Section~\ref{sec:MinEnergy}, 
we analyze the discrete energy minimization problem
providing the sampling points for interpolation and
lower bound for the corresponding discrete potential.
In Section~\ref{sec:Approx}, 
we propose an approximation formula by using the sampling points
and bound its error in each space $\boldsymbol{H}^{\infty}(\mathcal{D}_{d}, w)$.
In Section~\ref{sec:num_algo},
we present some results of numerical experiments. 
Finally, in Section~\ref{sec:concl}, we conclude this work.

\section{Mathematical Preliminaries}
\label{sec:prelim}

\subsection{Weight functions and weighted Hardy spaces}

Let $d$ be a positive real number, and let $\mathcal{D}_{d}$ be the strip region defined by 
$\mathcal{D}_{d} := \{ z \in \mathbf{C} \mid |\mathop{\mathrm{Im}} z | < d \}$. 
In order to specify the weight functions $w$ on $\mathcal{D}_{d}$ mathematically, 
we use the function space $B(\mathcal{D}_{d})$ of all functions $\zeta$ that are analytic on $\mathcal{D}_{d}$, 
such that
\begin{align}
\lim_{x \to \pm \infty} \int_{-d}^{d} |\zeta(x+\ii y)|\, \dd y = 0
\end{align} 
and 
\begin{align}
\lim_{y \to d-0}
\int_{-\infty}^{\infty} ( |\zeta(x+\ii y)| + |\zeta(x-\ii y)| )\, \dd x < \infty.
\end{align} 
Then, we regard $w: \mathcal{D}_{d} \to \mathbf{C}$ as a weight function
if $w$ satisfies the following assumption. 
\begin{assump}
\label{assump:w}
Function $w$ belongs to $B(\mathcal{D}_{d})$, 
does not vanish at any point in $\mathcal{D}_{d}$, 
and takes positive real values in $(0,1]$ on the real axis.
\end{assump}

\noindent
Furthermore,
throughout this work, 
we assume the log-concavity of the weight function $w$. 

\begin{assump}
\label{assump:w_convex}
Function $\log w$ is strictly concave on $\mathbf{R}$.
\end{assump}

For the weight function $w$ that satisfies Assumptions \ref{assump:w} and \ref{assump:w_convex}, 
we define the weighted Hardy space%
\footnote{
See 
\ifthenelse{\value{journal} = 0}{%
\cite[Chap.~10]{bib:Duren_Hardy_1970}
}{
\citet[Chap.~10]{bib:Duren_Hardy_1970}
}
as a reference for the Hardy spaces over general domains. 
}
on $\mathcal{D}_{d}$ by \eqref{eq:def_weighted_Hardy_rev}, i.e.,
\begin{align}
\notag 
\boldsymbol{H}^{\infty}(\mathcal{D}_{d}, w)
:=
\left\{
f : \mathcal{D}_{d} \to \mathbf{C} 
\ \left| \
f \text{ is analytic on } \mathcal{D}_{d} \text{ and } 
\| f \| < \infty 
\right.
\right\}, 
\end{align}
where
\begin{align}
\notag 
\| f \| := \sup_{z \in \mathcal{D}_{d}} \left| \frac{f(z)}{w(z)} \right|. 
\end{align}

\subsection{Optimal approximation}
\label{sec:DefOptApprox}

We provide a mathematical formulation for optimality of the approximation formula
in the space $\boldsymbol{H}^{\infty}(\mathcal{D}_{d}, w)$, 
with the weight function $w$ satisfying Assumptions \ref{assump:w} and \ref{assump:w_convex}. 
In this regard, for a given positive integer $n$, 
we first consider all the possible $n$-point interpolation formulas on $\mathbf{R}$ 
that can be applied to any function $f \in \boldsymbol{H}^{\infty}(\mathcal{D}_{d}, w)$. 
Subsequently, we choose a criterion that determines optimality of a formula in 
$\boldsymbol{H}^{\infty}(\mathcal{D}_{d}, w)$. 
Based on \cite{bib:Sugihara_NearOpt_2003}, 
we adopt the minimum worst error 
$E_{n}^{\mathrm{min}}(\boldsymbol{H}^{\infty}(\mathcal{D}_{d}, w))$ given by
\begin{align}
& E_{n}^{\mathrm{min}}(\boldsymbol{H}^{\infty}(\mathcal{D}_{d}, w)) \notag \\ 
& :=
\inf_{1 \leq l \leq n} 
\inf_{\begin{subarray}{c} m_{1}, \ldots , m_{l} \\ m_{1}+\cdots+m_{l} = n \end{subarray}}
\inf_{\begin{subarray}{c} a_{j} \in \mathcal{D}_{d} \\ \text{distinct} \end{subarray}}
\inf_{\phi_{jk}}
\left[
\sup_{\| f \| \leq 1}
\sup_{x \in \mathbf{R}}
\left|
f(x) - \sum_{j = 1}^{l} \sum_{k = 0}^{m_{j} - 1} f^{(k)}(a_{j})\, \phi_{jk}(x)
\right|
\right],
\label{eq:def_E_min} 
\end{align}
where $\phi_{jk}$ are analytic functions on $\mathcal{D}_{d}$. 
We regard a formula that attains this value as optimal. 

We can provide some characterizations of  
$E_{n}^{\mathrm{min}}(\boldsymbol{H}^{\infty}(\mathcal{D}_{d}, w))$. 
To achieve this, 
for a mutually distinct $n$-sequence $a = \{ a_{j} \}_{j = 1}^{n} \subset \mathbf{R}$, 
we introduce the following functions%
\footnote{
The function given by~\eqref{eq:trans_B_prod} is called the transformed Blaschke product. }
\begin{align}
T_{d}(x) 
&= \tanh \left( \frac{\pi}{4d} x \right), 
\label{eq:tanh} \\
B_{n}(x; a, \mathcal{D}_{d}) 
& = \prod_{j = 1}^{n} \frac{T_{d}(x) - T_{d}(a_{j})}{1 - \overline{T_{d}(a_{j})} \, T_{d}(x)}, 
\label{eq:trans_B_prod} \\
B_{n;k}(x; a, \mathcal{D}_{d}) 
& = \prod_{\begin{subarray}{c} 1 \leq j \leq n, \\ j \neq k \end{subarray} } 
\frac{T_{d}(x) - T_{d}(a_{j})}{1 - \overline{T_{d}(a_{j})} \, T_{d}(x)}, 
\label{eq:trans_B_k_prod} 
\end{align}
and the $n$-point interpolation formula
\begin{align}
L_{n}[a; f](x) = 
\sum_{k=1}^{n} f(a_{k}) 
\frac{B_{n:k}(x; a, \mathcal{D}_{d}) \, w(x)}{B_{n:k}(a_{k}; a, \mathcal{D}_{d}) \, w(a_{k})}
\frac{T_{d}'(x - a_{k})}{T_{d}'(0)}. 
\label{eq:Interp_Op}
\end{align}
Then, we describe characterizations of $E_{n}^{\mathrm{min}}(\boldsymbol{H}^{\infty}(\mathcal{D}_{d}, w))$, 
including the expression in~\eqref{eq:CharMinErr_Intro}, 
by the following proposition. 

\begin{prop}[%
\ifthenelse{\value{journal} = 0}{%
{\cite[Lemma 4.3 and its proof]{bib:Sugihara_NearOpt_2003}}%
}{
{\citet[Lemma 4.3 and its proof]{bib:Sugihara_NearOpt_2003}}%
}]
\label{thm:CharMinErr}
Let $a = \{ a_{j} \}_{j = 1}^{n} \subset \mathbf{R}$ be a mutually distinct sequence. 
Then, we have the following error estimate of the formula in~\eqref{eq:Interp_Op}:
\begin{align}
E_{n}^{\mathrm{min}}(\boldsymbol{H}^{\infty}(\mathcal{D}_{d}, w)) 
& \leq
\sup_{
\begin{subarray}{c}
f \in \boldsymbol{H}^{\infty}(\mathcal{D}_{d}, w) \\
\| f \| \leq 1
\end{subarray}
} 
\left(
\sup_{x \in \mathbf{R}}
\left| 
f(x) - L_{n}[a; f](x) 
\right|
\right) 
\label{eq:CharApprox_Ineq_1st} \\
& \leq 
\sup_{x \in \mathbf{R}}
\left| 
B_{n}(x; a, \mathcal{D}_{d}) \, w(x)
\right|.
\label{eq:CharApprox_Ineq_2nd} 
\end{align}
Furthermore, if we take the infimum over all the $n$-sequences $a$ in the above inequalities,  then
each of them becomes an equality. 
\begin{align}
E_{n}^{\mathrm{min}}(\boldsymbol{H}^{\infty}(\mathcal{D}_{d}, w))
& = 
\inf_{a_{j} \in \mathbf{R}} 
\left[
\sup_{
\begin{subarray}{c}
f \in \boldsymbol{H}^{\infty}(\mathcal{D}_{d}, w) \label{eq:CharApprox_Eq} \\
\| f \| \leq 1
\end{subarray}
} 
\left(
\sup_{x \in \mathbf{R}}
\left| 
f(x) - L_{n}[a; f](x) 
\right|
\right)	
\right] \\
& = 
\inf_{a_{j} \in \mathbf{R}} 
\left[
\sup_{x \in \mathbf{R}}
\left| 
B_{n}(x; a, \mathcal{D}_{d}) \, w(x)
\right|
\right]. \label{eq:CharMinErr}
\end{align}
\end{prop}

\noindent
Proposition~\ref{thm:CharMinErr} indicates that 
the interpolation formula $L_{n}[a; f](x)$
provides an explicit form of an optimal approximation formula 
if there exists an $n$-sequence $a = a^{\ast}$ that attains the infimum in~\eqref{eq:CharMinErr}. 
Since
\[
\frac{T_{d}(x) - T_{d}(a_{j})}{1 - \overline{T_{d}(a_{j})} \, T_{d}(x)} = T_{d}(x - a_{j})
\]
for $a_{j} \in \mathbf{R}$ and $x \in \mathbf{R}$, 
the expression in \eqref{eq:CharMinErr} can be rewritten in the form
\begin{align}
\inf_{a_{j} \in \mathbf{R}} 
\left[
\sup_{x \in \mathbf{R}}
\left| 
\left(
\prod_{j=1}^{n} T_{d}(x - a_{j})
\right) 
w(x)
\right|
\right]. \notag
\end{align}
Therefore,
as far as $a = \{ a_{j} \}_{j=1}^{n}$ is concerned, 
we can consider the following equivalent alternative: 
\begin{align}
\inf_{a_{j} \in \mathbf{R}} 
\left[
\sup_{x \in \mathbf{R}}
\left(
\sum_{j=1}^{n} \log |T_{d}(x - a_{j})|
+ \log w(x)
\right)
\right].
\label{eq:CharMinErr_rew}
\end{align}

To deal with the optimization problem corresponding to~\eqref{eq:CharMinErr_rew}, 
we introduce the following notation:
\begin{align}
& K(x)  = - \log |T_{d}(x)| 
\quad \left( = -\log \left| \tanh \left( \frac{\pi}{4d} x \right) \right| \right), 
\label{eq:SettingK} \\
& Q(x)  = - \log w(x). 
\label{eq:SettingQ} 
\end{align}
Furthermore, for an integer $n \geq 2$, let 
\begin{align}
\mathcal{R}_{n} = 
\{ (a_{1}, \ldots , a_{n}) \in \mathbf{R}^{n} \mid a_{1} < \cdots < a_{n} \}
\label{eq:ordered_n_points}
\end{align}
be the set of mutually distinct $n$-point configurations in $\mathbf{R}$. 
Then,  by using the function defined by
\begin{align}
& U_{n}^{\mathrm{D}}(a;x) = \sum_{i=1}^{n} K(x - a_{i}), \qquad x \in \mathbf{R}
\label{eq:def_D_pot}
\end{align}  
for $a = (a_{1}, \ldots , a_{n}) \in \mathcal{R}_{n}$,
we can formulate the optimization problem corresponding to~\eqref{eq:CharMinErr_rew}
as follows: 
\begin{align}
(\mathrm{D})
\qquad 
\text{maximize} 
\quad 
\inf_{x \in \mathbf{R}} \left( U_{n}^{\mathrm{D}}(a;x) + Q(x) \right)
\quad
\text{subject to} 
\quad 
a \in \mathcal{R}_{n}. 
\label{eq:OptPot}
\end{align}

Problem (D) in~\eqref{eq:OptPot} is closely related to potential theory. 
In fact, function $K(x - y)$ of $x,y \in \mathbf{R}^{2}$ 
is the kernel function derived from the Green function of $\mathcal{D}_{d}$: 
\begin{align}
g_{\mathcal{D}_{d}}(z_{1},z_{2}) = - \log \left| \frac{T_{d}(z_{1}) - T_{d}(z_{2})}{1 - \overline{T_{d}(z_{2})} \, T_{d}(z_{1})} \right|
\label{eq:Green_D_d}
\end{align}
in the special case that $(z_{1}, z_{2}) = (x,y) \in \mathbf{R}^{2}$. 
Therefore, the function $U_{n}^{\mathrm{D}}(a;x)$ is the Green potential 
for the discrete measure $\sum_{i=1}^{n} \delta_{a_{i}}$, 
where $\delta_{a_{i}}$ is the Dirac measure centered at $a_{i}$. 
Because some fundamental results about the Green potential
can be used as good references to deal with Problem (D), 
we describe them below in Section~\ref{sec:PotentialTheory}. 

\subsection{Fundamentals of potential theory}
\label{sec:PotentialTheory}

For a positive integer $n$, let $\mathcal{M}(\mathbf{R}, n)$ 
be the set of all Borel measures $\mu$ on $\mathbf{R}$ with $\mu(\mathbf{R}) = n$, 
and let $\mathcal{M}_{\mathrm{c}}(\mathbf{R}, n)$ be the set of measures 
$\mu \in \mathcal{M}(\mathbf{R}, n)$ with a compact support. 
In particular, 
for a sequence $a \in \mathcal{R}_{n}$, 
the discrete measure $\sum_{i=1}^{n} \delta_{a_{i}}$
belongs to $\mathcal{M}_{\mathrm{c}}(\mathbf{R}, n)$. 
For $\mu \in \mathcal{M}(\mathbf{R}, n)$, 
we define potential $U_{n}^{\mathrm{C}}(\mu; x)$ and energy $I_{n}^{\mathrm{C}}(\mu)$ by
\begin{align}
& U_{n}^{\mathrm{C}}(\mu; x)
= 
\int_{\mathbf{R}} K(x - y) \, \mathrm{d}\mu(y), 
\label{eq:def_C_pot} \\
& I_{n}^{\mathrm{C}}(\mu)
= 
\int_{\mathbf{R}} \int_{\mathbf{R}} K(x - y) \, \mathrm{d}\mu(y) \mathrm{d}\mu(x)
+ 2 \int_{\mathbf{R}} Q(x) \, \mathrm{d}\mu(x), 
\label{eq:def_C_ene} 
\end{align}
respectively. 
According to \eqref{eq:SettingK} and \eqref{eq:Green_D_d}, 
these are the Green potential and energy 
in the case that the domain of the Green function is $\mathcal{D}_{d}$
and that of the external field $Q$ is $\mathbf{R}$. 
By using standard techniques in potential theory, 
we can show the following fundamental theorems. 

\begin{thm}
\label{thm:Green_basic}
On Assumptions~\ref{assump:w} and~\ref{assump:w_convex}, the following hold true: 
\begin{enumerate}
\item
The energy $I_{n}^{\mathrm{C}}(\mu)$ has a unique minimizer $\mu_{n}^{\ast} \in \mathcal{M}(\mathbf{R}, n)$
with $I_{n}^{\mathrm{C}}(\mu_{n}^{\ast}) < \infty$. 
Moreover, $\mu_{n}^{\ast}$ has finite energy: 
\[
\int_{\mathbf{R}} U_{n}^{\mathrm{C}}(\mu_{n}^{\ast}; x) \, \mathrm{d}\mu_{n}^{\ast}(x) < \infty. 
\]

\item
The support $\mathop{\mathrm{supp}} \mu_{n}^{\ast}$ is the compact subset of $\mathbf{R}$, 
i.e., $\mu_{n}^{\ast} \in \mathcal{M}_{\mathrm{c}}(\mathbf{R}, n)$. 
More precisely, 
$\mathop{\mathrm{supp}} \mu_{n}^{\ast} \subset \{ x \in \mathbf{R} \mid Q(x) \leq N_{n} \}$ 
holds true for some $N_{n}$. 

\item
Let constant $F_{K,Q}^{\mathrm{C}}(n)$ be defined by
\begin{align}
F_{K,Q}^{\mathrm{C}}(n)
=
I_{n}^{\mathrm{C}}(\mu_{n}^{\ast}) - \int_{\mathbf{R}} Q(x) \, \mathrm{d}\mu_{n}^{\ast}(x). 
\label{eq:def_FC}
\end{align}
Then, we have
\begin{align}
& U_{n}^{\mathrm{C}}(\mu_{n}^{\ast}; x) + Q(x) \geq \frac{F_{K,Q}^{\mathrm{C}}(n)}{n} \qquad \forall x \in \mathbf{R}, 
\label{eq:pot_geq_F} \\
& U_{n}^{\mathrm{C}}(\mu_{n}^{\ast}; x) + Q(x) = \frac{F_{K,Q}^{\mathrm{C}}(n)}{n} \qquad \forall x \in \mathop{\mathrm{supp}} \mu_{n}^{\ast}. 
\label{eq:pot_eq_F} 
\end{align}
\end{enumerate}
\end{thm}

\begin{proof}
This theorem is a specialized version of Theorems 2.1 and 2.2 in \cite{bib:LevinLubinsky_2001}. 
In fact, if we set 
$G = \mathcal{D}_{d}$, $E = \mathbf{R}$ and $Q(x) = - \log w(x)$, 
then $Q$ is admissible on $\mathbf{R}$ and the assumptions of these theorems are satisfied. 
In particular, the assertion $\displaystyle \lim_{x \to \pm \infty, \, x \in \mathbf{R}} Q(x) = \infty$
holds true owing to Assumption~\ref{assump:w}.  
Therefore, the proof of this theorem is straightforward and omitted here. 
\end{proof}

\begin{thm}
\label{thm:PotUpBoundFC}
On Assumptions \ref{assump:w} and \ref{assump:w_convex}, 
for any $\mu \in \mathcal{M}_{\mathrm{c}}(\mathbf{R}, n)$, 
there exists $x \in \mathbf{R}$ such that
\begin{align}
U_{n}^{\mathrm{C}}(\mu; x) + Q(x) \leq \frac{F_{K,Q}^{\mathrm{C}}(n)}{n}.
\label{eq:any_mu_leq_F}
\end{align}
\end{thm}

\begin{proof}
Because this theorem is an analogue of the first half of Theorem I.3.1 in \cite{bib:SaffTotik_LogPotExtField_1997}, 
this proof is basically similar to that of the same theorem. 
Suppose that 
\[
U_{n}^{\mathrm{C}}(\mu; x) + Q(x) \geq L 
\qquad 
\text{for any } x \in \mathbf{R}
\]
holds for some $L$. 
Then, by Inequality~\eqref{eq:pot_eq_F} in Theorem~\ref{thm:Green_basic}, we have
\begin{align}
& U_{n}^{\mathrm{C}}(\mu; x) - U_{n}^{\mathrm{C}}(\mu^{\ast}; x) \geq L - \frac{F_{K,Q}^{\mathrm{C}}(n)}{n} \notag \\
& \iff 
U_{n}^{\mathrm{C}}(\mu; x) \geq U_{n}^{\mathrm{C}}(\mu^{\ast}; x) + L - \frac{F_{K,Q}^{\mathrm{C}}(n)}{n} 
\label{eq:for_applying_POD}
\end{align}
for any $x \in \mathop{\mathrm{supp}} \mu^{\ast}$. 
Then, by the principle of domination,  
Equality~\eqref{eq:for_applying_POD} holds for all $z \in \mathcal{D}_{d}$. 
By letting $z \to z_{0} \in \partial \mathcal{D}_{d}$, we have
\[
\frac{F_{K,Q}^{\mathrm{C}}(n)}{n} \geq L. 
\]
Therefore, 
there exists $x \in \mathbf{R}$ such that 
\[
U_{n}^{\mathrm{C}}(\mu; x) + Q(x) \leq \frac{F_{K,Q}^{\mathrm{C}}(n)}{n}, 
\]
which proves the theorem. 
\end{proof}

Then, 
according to Inequalities~\eqref{eq:pot_geq_F} and~\eqref{eq:any_mu_leq_F}, 
we can obtain the following theorem. 

\begin{thm}
\label{thm:equi_meas_opt}
On Assumptions \ref{assump:w} and \ref{assump:w_convex}, 
the minimizer $\mu_{n}^{\ast}$ of $I_{n}^{\mathrm{C}}$ yields a solution of the optimization problem
\begin{align}
(\mathrm{C})
\qquad 
\text{maximize}
\quad 
\inf_{x \in \mathbf{R}} (U_{n}^{\mathrm{C}}(\mu; x) + Q(x)) 
\quad
\text{subject to}
\quad 
\mu \in \mathcal{M}_{\mathrm{c}}(\mathbf{R}, n).
\label{eq:OptPotCont}
\end{align}
\end{thm}

\begin{proof}
This is a direct consequence of Theorems~\ref{thm:Green_basic} and~\ref{thm:PotUpBoundFC}.
\end{proof}

\section{Minimization of discrete energy}
\label{sec:MinEnergy}

Our ideal goal is finding an optimal solution $a^{\dagger} \in \mathcal{R}_{n}$ of Problem (D) defined in~\eqref{eq:OptPot}
and proposing an optimal interpolation formula $L_{n}[a^{\dagger}; f]$. 
However, it is difficult to solve Problem (D) directly. 
Therefore, with a view to a discrete analogue of Theorem~\ref{thm:equi_meas_opt}, 
we define discrete energy $I_{n}^{\mathrm{D}}(a)$ as
\begin{align}
I_{n}^{\mathrm{D}}(a) 
= 
\sum_{i=1}^{n} \sum_{\begin{subarray}{c} j=1 \\ j \neq i \end{subarray}}^{n} K(a_{i} - a_{j}) 
+ \frac{2(n-1)}{n} \sum_{i=1}^{n} Q(a_{i})
\label{eq:def_D_ene}
\end{align}
for $a = (a_{1}, \ldots, a_{n}) \in \mathcal{R}_{n}$, and consider its minimization.

In this section, we show that 
$I_{n}^{\mathrm{D}}$ is easily tractable owing to Assumptions~\ref{assump:w} and~\ref{assump:w_convex}, 
and that its minimizer is an approximate solution of Problem (D). 
First, we confirm the basic properties of $K$ and $Q$. 

\begin{prop}
\label{thm:basic_K}
The function $K$ defined by~\eqref{eq:SettingK} is positive, even, and convex as a function on $\mathbf{R} \setminus \{ 0 \}$. 
Furthermore, it satisfies $\displaystyle \lim_{x \to \pm 0} K(x) = \infty$. 
\end{prop}

\begin{prop}
\label{thm:basic_Q}
On Assumptions~\ref{assump:w} and~\ref{assump:w_convex}, 
the function $Q$ defined by~\eqref{eq:SettingQ} 
is twice differentiable and strictly convex on $\mathbf{R}$. 
Therefore, we have $Q''(x) > 0$ for any $x \in \mathbf{R}$. 
\end{prop}

\noindent
Because these propositions can be easily proved, we omit their proofs. 
Next, we show the solvability of the minimization of $I_{n}^{\mathrm{D}}$. 

\begin{thm}
On Assumptions~\ref{assump:w} and~\ref{assump:w_convex}, 
the energy $I_{n}^{\mathrm{D}}$ is convex in $\mathcal{R}_{n}$, and 
there is a unique minimizer of $I_{n}^{\mathrm{D}}$ in $\mathcal{R}_{n}$. 
\end{thm}

\begin{proof}
Let $H_{n}(a)$ be the Hessian of $I_{n}^{\mathrm{D}}$ at $a \in \mathcal{R}_{n}$. 
First, we show that $H_{n}(a)$ is positive definite for any $a \in \mathcal{R}_{n}$. 
Because we have
\begin{align}
\frac{\partial}{\partial a_{\ell}} I_{n}^{\mathrm{D}}(a) 
= 
2 \sum_{\begin{subarray}{c} j = 1 \\ j \neq \ell \end{subarray}}^{n} K'(a_{\ell} - a_{j}) + \frac{2(n-1)}{n} Q'(a_{\ell}), 
\label{eq:first_deriv_I_D}
\end{align}
the $(k, \ell)$-component of $H_{n}(a)$ is given by 
\begin{align}
\frac{\partial^{2}}{\partial a_{k} \partial a_{\ell}} I_{n}^{\mathrm{D}}(a) 
= 
\begin{cases}
2 \sum_{\begin{subarray}{c} j = 1 \\ j \neq \ell \end{subarray}}^{n} K''(a_{\ell} - a_{j}) + \frac{2(n-1)}{n} Q''(a_{\ell}) & (k = \ell) \\ 
- 2 K''(a_{\ell} - a_{k}) & (k \neq \ell). 
\end{cases}
\end{align}
Because $K$ is convex and $Q$ is strictly convex, 
the diagonal components of $H_{n}(a)$ are positive. 
Furthermore,  
$H_{n}(a)$ is strictly diagonally dominant because
\begin{align*}
\sum_{\begin{subarray}{c} k = 1 \\ k \neq \ell \end{subarray}}^{n} \left| -2 K''(a_{\ell} - a_{k}) \right|
& = 
2 \sum_{\begin{subarray}{c} k = 1 \\ k \neq \ell \end{subarray}}^{n} K''(a_{\ell} - a_{k}) 
< 
2 \sum_{\begin{subarray}{c} k = 1 \\ k \neq \ell \end{subarray}}^{n} K''(a_{\ell} - a_{k}) + \frac{2(n-1)}{n} Q''(a_{\ell}). 
\end{align*}
Therefore $H_{n}(a)$ is positive definite \cite[Corollary 7.2.3]{bib:HJ_Matrix_1990}, 
which implies $I_{n}^{\mathrm{D}}$ is a strictly convex function on $\mathcal{R}_{n}$. 

Next, we show the existence of a unique minimizer of $I_{n}^{\mathrm{D}}$ in $\mathcal{R}_{n}$. 
Because $Q(x) \to \infty$ as $x \to \pm \infty$, 
there exists $r_{n} > 0$ such that 
\begin{align*}
|x| > r_{n} \ \Rightarrow \ \frac{2(n-1)}{n} Q(x) > I_{n}^{\mathrm{D}}(1, 2, \ldots, n). 
\end{align*}
Then, for $(a_{1}, \ldots , a_{n}) \in \mathcal{R}_{n}$ with $\max\{ |a_{1}|, |a_{n}| \} > r_{n}$, 
we have 
\begin{align*}
I_{n}^{\mathrm{D}}(a_{1}, \ldots, a_{n}) > \frac{2(n-1)}{n} \max\{ Q(a_{1}), Q(a_{n}) \} > I_{n}^{\mathrm{D}}(1,2, \ldots, n). 
\end{align*}
Therefore, it suffices to consider the minimization of $I_{n}^{\mathrm{D}}(a)$ in the bounded set
$\tilde{\mathcal{R}}_{n} = 
\{ (a_{1}, \ldots, a_{n}) \in \mathbf{R}^{n} \mid -r_{n} \leq a_{1} < \cdots < a_{n} \leq r_{n} \}$. 
This minimization is equivalent to the maximization of $\exp(-I_{n}^{\mathrm{D}}(a))$ on $\tilde{\mathcal{R}}_{n}$. 
Because
\begin{align*}
\lim_{\begin{subarray}{c} a_{i} \to a_{j} \\ (a_{1}, \ldots, a_{n}) \in \tilde{\mathcal{R}}_{n} \end{subarray}} \exp(-I_{n}^{\mathrm{D}}(a)) = 0
\qquad (j = i-1 \text{ or } j = i+1), 
\end{align*}
the function $J_{n}^{\mathrm{D}}(a)$ defined by 
\begin{align*}
J_{n}^{\mathrm{D}}(a) = 
\begin{cases}
\exp(-I_{n}^{\mathrm{D}}(a)) & (a \in \tilde{\mathcal{R}}_{n}), \\
0  & (a \in \mathop{\mathrm{cl}}(\tilde{\mathcal{R}}_{n}) \setminus \tilde{\mathcal{R}}_{n})
\end{cases}
\end{align*}
is a continuous function on $\mathop{\mathrm{cl}}(\tilde{\mathcal{R}}_{n})$, 
where ``$\mathop{\mathrm{cl}}$'' denotes the closure of a set. 
Therefore, there exists a maximizer of $J_{n}^{\mathrm{D}}(a)$ in $\mathop{\mathrm{cl}}(\tilde{\mathcal{R}}_{n})$. 
Actually any maximizer is in $\tilde{\mathcal{R}}_{n}$ because the maximum value is positive. 
Hence the minimizer of $I_{n}^{\mathrm{D}}(a)$ exists in $\tilde{\mathcal{R}}_{n}$, 
which is unique because $I_{n}^{\mathrm{D}}(a)$ is strictly convex. 
\end{proof}

Let $a^{\ast} = (a_{1}^{\ast}, \ldots , a_{n}^{\ast}) \in \mathcal{R}_{n}$ be the minimizer of $I_{n}^{\mathrm{D}}$, 
and let $F_{K, Q}^{\mathrm{D}}(n)$ be the number defined by
\begin{align}
F_{K, Q}^{\mathrm{D}}(n) = I_{n}^{\mathrm{D}}(a^{\ast}) - \frac{n-1}{n} \sum_{i=1}^{n} Q(a_{i}^{\ast}), 
\label{eq:def_FD}
\end{align}
which is a discrete analogue of $F_{K, Q}^{\mathrm{C}}(n)$ in~\eqref{eq:def_FC}. 
Then, we can show a discrete analogue of Inequality~\eqref{eq:pot_geq_F}, 
which indicates that $a^{\ast}$ is an approximate solution of Problem (D). 

\begin{thm}
\label{thm:PotLowBoundFD}
Let $a^{\ast} \in \mathcal{R}_{n}$ be the minimizer of $I_{n}^{\mathrm{D}}$. 
On Assumptions~\ref{assump:w} and~\ref{assump:w_convex}, 
we have
\begin{align}
U_{n}^{\mathrm{D}}(a^{\ast};x) + Q(x) \geq \frac{F_{K, Q}^{\mathrm{D}}(n)}{n-1}
\qquad \text{for any $x \in \mathbf{R}$.}
\label{eq:PotLowBoundFD}
\end{align}
\end{thm}

\begin{proof}
First, we show that
\begin{align}
\sum_{\begin{subarray}{c} j = 1 \\ j \neq k \end{subarray}}^{n} K(x - a_{j}^{\ast}) + \frac{n-1}{n} Q(x)
\geq  
\sum_{\begin{subarray}{c} j = 1 \\ j \neq k \end{subarray}}^{n} K(a_{k}^{\ast} - a_{j}^{\ast}) + \frac{n-1}{n} Q(a_{k}^{\ast})
\label{eq:PartPotLowBound}
\end{align}
for any $x \in \mathbf{R}$ and $k = 1,\ldots, n$.
Suppose that Inequality~\eqref{eq:PartPotLowBound} does not hold for some $x$ and $k$:
\begin{align} 
\sum_{\begin{subarray}{c} j = 1 \\ j \neq k \end{subarray}}^{n} K(x - a_{j}^{\ast}) + \frac{n-1}{n} Q(x)
<
\sum_{\begin{subarray}{c} j = 1 \\ j \neq k \end{subarray}}^{n} K(a_{k}^{\ast} - a_{j}^{\ast}) + \frac{n-1}{n} Q(a_{k}^{\ast}). 
\label{eq:not_PartPotLowBound}
\end{align}
Then, by multiplying both sides of \eqref{eq:not_PartPotLowBound} by $2$ and adding 
\begin{align*}
\sum_{\begin{subarray}{c} i = 1 \\ i \neq k \end{subarray}}^{n}
\sum_{\begin{subarray}{c} j = 1 \\ j \neq i, k \end{subarray}}^{n} K(a_{i}^{\ast} - a_{j}^{\ast}) 
+ \frac{2(n-1)}{n} \sum_{\begin{subarray}{c} i = 1 \\ i \neq k \end{subarray}}^{n} Q(a_{i}^{\ast})
\end{align*}
to them, we have 
\begin{align*}
I_{n}^{\mathrm{D}}(b) = 
2\sum_{\begin{subarray}{c} j = 1 \\ j \neq k \end{subarray}}^{n} K(x - a_{j}^{\ast})
+
\sum_{\begin{subarray}{c} i = 1 \\ i \neq k \end{subarray}}^{n}
\sum_{\begin{subarray}{c} j = 1 \\ j \neq i, k \end{subarray}}^{n} K(a_{i}^{\ast} - a_{j}^{\ast}) 
+ \frac{2(n-1)}{n} \left( Q(x) + \sum_{\begin{subarray}{c} i = 1 \\ i \neq k \end{subarray}}^{n} Q(a_{i}^{\ast}) \right)
< 
I_{n}^{\mathrm{D}}(a^{\ast}), 
\end{align*}
where $b = (b_{1}, \ldots, b_{n}) \in \mathcal{R}_{n}$ is the $n$ point configuration 
obtained by sorting $(a_{1}^{\ast}, \ldots, a_{k-1}^{\ast}, x , a_{k+1}^{\ast}, \ldots, a_{n}^{\ast})$. 
Thus we have Inequality~\eqref{eq:PartPotLowBound} by contradiction. 

Then, summing up both sides of Inequality~\eqref{eq:PartPotLowBound} for $k = 1,\ldots, n$, we have
\begin{align}
& \sum_{k=1}^{n} 
\left( \sum_{j=1}^{n} K(x - a_{j}^{\ast}) - K(x - a_{k}^{\ast}) \right)
+ (n-1) Q(x) \geq I_{n}^{\mathrm{D}}(a^{\ast}) - \frac{n-1}{n} \sum_{i = 1}^{n} Q(a_{i}^{\ast}) \notag \\
& \iff
(n-1) \left( \sum_{j=1}^{n} K(x - a_{j}^{\ast}) + Q(x) \right) \geq F_{K, Q}^{\mathrm{D}}(n), \notag
\end{align}
which is equivalent to Inequality~\eqref{eq:PotLowBoundFD}. 
\end{proof}

Let $P_{n}$ be the optimal value of Problem (D) in~\eqref{eq:OptPot}:
\begin{align}
P_{n} = 
\sup_{a_{i} \in \mathbf{R}}
\left(
\inf_{x \in \mathbf{R}} \left( U_{n}^{\mathrm{D}}(a;x) + Q(x) \right)
\right).
\label{eq:OptVal}
\end{align}
Then, 
by using Theorems~\ref{thm:PotUpBoundFC} and~\ref{thm:PotLowBoundFD}, 
we can obtain lower and upper bounds of $P_{n}$. 

\begin{thm}
\label{thm:upper_lower_mlogE}
On Assumptions~\ref{assump:w} and~\ref{assump:w_convex}, 
we have
\begin{align}
\frac{F_{K,Q}^{\mathrm{D}}(n)}{n} 
\leq 
P_{n} 
\leq 
\frac{F_{K,Q}^{\mathrm{C}}(n)}{n}, 
\end{align}
which implies that the minimizer $a^{\ast} \in \mathcal{R}_{n}$ of $I_{n}^{\mathrm{D}}$
is an approximate solution of Problem (D) in~\eqref{eq:OptPot}, 
whose approximation rate is bounded by $F_{K,Q}^{\mathrm{D}}(n)/F_{K,Q}^{\mathrm{C}}(n)$. 
\end{thm}

\begin{proof}
By Theorem~\ref{thm:PotLowBoundFD}, 
we can obtain the lower bound as follows: 
\begin{align}
P_{n}
\geq 
\inf_{x \in \mathbf{R}} (U_{n}^{\mathrm{D}}(a^{\ast}; x)+Q(x)) 
\geq 
\frac{F_{K,Q}^{\mathrm{D}}(n)}{n}. 
\label{eq:LowBdPn}
\end{align}

On the other hand, we can provide the upper bound by Theorem~\ref{thm:PotUpBoundFC}. 
In fact, for any $a \in \mathcal{R}_{n}$, there exists $x \in \mathbf{R}$ such that
\begin{align}
U_{n}^{\mathrm{D}}(a;x) + Q(x) \leq \frac{F_{K,Q}^{\mathrm{C}}(n)}{n}
\end{align}
because we can consider the special case 
$\mu = \sum_{i=1}^{n} \delta_{a_{i}} \in \mathcal{M}_{\mathrm{c}}(\mathbf{R}, n)$ in Theorem~\ref{thm:PotUpBoundFC}.
Then, we have
\begin{align}
P_{n}
\leq 
\frac{F_{K,Q}^{\mathrm{C}}(n)}{n}. 
\label{eq:UpBdPn}
\end{align}
\end{proof}

\section{Design of approximation formulas}
\label{sec:Approx}

\subsection{Proposed formula and its error estimate}

By using the minimizer $a^{\ast} \in \mathcal{R}_{n}$ of $I_{n}^{\mathrm{D}}$, 
we propose the approximation formula $L_{n}[a^{\ast}; f]$ for $f \in \boldsymbol{H}^{\infty}(\mathcal{D}_{d}, w)$, 
where $L_{n}[a; f]$ is defined by~\eqref{eq:Interp_Op}. 
That is, $L_{n}[a^{\ast}; f]$ is written in the form 
\begin{align}
L_{n}[a^{\ast}; f](x) = 
\sum_{k=1}^{n} f(a_{k}^{\ast}) 
\frac{B_{n:k}(x; a^{\ast}, \mathcal{D}_{d}) \, w(x)}{B_{n:k}(a_{k}^{\ast}; a^{\ast}, \mathcal{D}_{d}) \, w(a_{k}^{\ast})}
\frac{T_{d}'(x - a_{k}^{\ast})}{T_{d}'(0)}.
\label{eq:ProposedApproxFormula}
\end{align}
We can provide an error estimate of this formula. 

\begin{thm}
\label{thm:ErrorEstimate}
Let $a^{\ast} \in \mathcal{R}_{n}$ be the minimizer of the discrete energy $I_{n}^{\mathrm{D}}$ 
and let $L_{n}[a^{\ast}; f]$ be the approximation formula for $f \in \boldsymbol{H}^{\infty}(\mathcal{D}_{d}, w)$
given by~\eqref{eq:ProposedApproxFormula}. 
On Assumptions~\ref{assump:w} and~\ref{assump:w_convex}, we have
\begin{align}
\sup_{
\begin{subarray}{c}
f \in \boldsymbol{H}^{\infty}(\mathcal{D}_{d}, w) \\
\| f \| \leq 1
\end{subarray}
} 
\left(
\sup_{x \in \mathbf{R}}
\left| 
f(x) - L_{n}[a^{\ast}; f](x) 
\right|
\right) 
\leq
\exp \left( - \frac{F_{K,Q}^{\mathrm{D}}(n)}{n} \right). 
\label{eq:ErrorEstimate}
\end{align}
\end{thm}

\begin{proof}
From Inequality~\eqref{eq:CharApprox_Ineq_2nd} in Proposition~\ref{thm:CharMinErr} and Theorem~\ref{thm:PotLowBoundFD}, 
we have
\begin{align}
\sup_{
\begin{subarray}{c}
f \in \boldsymbol{H}^{\infty}(\mathcal{D}_{d}, w) \\
\| f \| \leq 1
\end{subarray}
} 
\left(
\sup_{x \in \mathbf{R}}
\left| 
f(x) - L_{n}[a^{\ast}; f](x) 
\right|
\right) 
& \leq 
\sup_{x \in \mathbf{R}}
\left| 
B_{n}(x; a^{\ast}, \mathcal{D}_{d}) \, w(x)
\right|. \notag \\
& = 
\sup_{x \in \mathbf{R}}
\exp \left( - U_{n}^{\mathrm{D}}(a^{\ast};x) - Q(x) \right) \notag \\
& \leq 
\exp \left( - \frac{F_{K, Q}^{\mathrm{D}}(n)}{n} \right). \notag 
\end{align}
\end{proof}

\begin{rem}
From the inequalities in Theorem~\ref{thm:upper_lower_mlogE}, 
we have
\begin{align}
\exp \left( - \frac{F_{K,Q}^{\mathrm{C}}(n)}{n} \right)
\leq 
E_{n}^{\mathrm{min}}(\boldsymbol{H}^{\infty}(\mathcal{D}_{d}, w))
\leq 
\exp \left( - \frac{F_{K,Q}^{\mathrm{D}}(n)}{n}  \right). 
\end{align}
Therefore, we can regard the proposed formula in~\eqref{eq:ProposedApproxFormula} 
as nearly optimal if the exponents 
$F_{K,Q}^{\mathrm{C}}(n)/n$ and $F_{K,Q}^{\mathrm{D}}(n)/n$ are sufficiently close. 
However, we have not found their exact orders. 
Their precise estimates will be considered in future work. 
As a preliminary attempt for the estimate, 
we provide an upper bound of the difference $F_{K,Q}^{\mathrm{C}}(n) - F_{K,Q}^{\mathrm{D}}(n)$
by using the separation distance $h_{a^{\ast}}$ given by~\eqref{eq:sep_dist_a_ast}
in Appendix~\ref{sec:diff_FC_FD}. 
\end{rem}

\subsection{Barycentric forms of the proposed formula}

We can obtain some alternative forms of Formula~\eqref{eq:ProposedApproxFormula}
to reduce its computational cost. 
As such alternatives, 
we derive analogues of the barycentric formulas for 
the Lagrange interpolation 
\ifthenelse{\value{journal} = 0}{%
\cite{bib:BerrutTrefethen_bary_2004, bib:Trefethen_APTP_2013}. 
}{
\citep{bib:BerrutTrefethen_bary_2004, bib:Trefethen_APTP_2013}. 
}
Because they are categorized into two types, 
we derive two analogues. 
As shown below, 
the second one is derived only approximately. 

We begin with the first type. 
By letting 
\begin{align}
\lambda_{k}^{\ast} 
= \frac{1}{B_{n:k}(a_{k}^{\ast}; a^{\ast}, \mathcal{D}_{d})}
= \frac{1}{\prod_{j \neq k} T_{d}(a_{k}^{\ast} - a_{j}^{\ast})}
\qquad (k = 1,\ldots , n), 
\label{eq:def_lambda}
\end{align}
we have 
\begin{align}
L_{n}[a^{\ast}; f](x) 
& = 
w(x) 
\sum_{k=1}^{n} 
\frac{f(a_{k}^{\ast})}{w(a_{k}^{\ast})} \, 
\frac{\lambda_{k}^{\ast}}{T_{d}'(0)} \, 
B_{n:k}(x; a^{\ast}, \mathcal{D}_{d}) \, 
T_{d}'(x - a_{k}^{\ast}) \notag \\
& = 
w(x) 
\sum_{k=1}^{n} 
\frac{f(a_{k}^{\ast})}{w(a_{k}^{\ast})} \, 
\frac{\lambda_{k}^{\ast}}{T_{d}'(0)} 
\left( \prod_{j = 1}^{n} T_{d}(x - a_{j}^{\ast}) \right) 
\frac{T_{d}'(x - a_{k}^{\ast})}{T_{d}(x - a_{k}^{\ast})} \notag \\
& = 
w(x) \, B_{n}(x; a^{\ast}, \mathcal{D}_{d})
\sum_{k=1}^{n} 
\frac{f(a_{k}^{\ast})}{w(a_{k}^{\ast})} \, 
\frac{\lambda_{k}^{\ast}}{T_{d}'(0)} \, 
\frac{T_{d}'(x - a_{k}^{\ast})}{T_{d}(x - a_{k}^{\ast})} \notag \\
& = 
w(x) \, B_{n}(x; a^{\ast}, \mathcal{D}_{d})
\sum_{k=1}^{n} 
\frac{\lambda_{k}^{\ast}}{S_{d}(x - a_{k}^{\ast})} \, 
\frac{f(a_{k}^{\ast})}{w(a_{k}^{\ast})}, 
\label{eq:FirstBary}
\end{align}
where 
\begin{align}
S_{d}(x) = 
\frac{T_{d}'(0) \, T_{d}(x)}{T_{d}'(x)} 
= 
\frac{1}{2} \sinh \left( \frac{\pi}{2d} x \right). 
\label{eq:def_Sd}
\end{align}
Then, we regard Formula~\eqref{eq:FirstBary} as the analogue of the first barycentric formula. 

Next, we consider the second type. 
By letting $f = w$ in~\eqref{eq:FirstBary},  we have 
\begin{align}
L_{n}[a^{\ast}; w](x)
=
w(x) B_{n}(x; a^{\ast}, \mathcal{D}_{d})
\sum_{k=1}^{n} 
\frac{\lambda_{k}^{\ast}}{S_{d}(x - a_{k}^{\ast})}. 
\notag
\end{align}
Then, by noting that $L_{n}[a^{\ast}; w](x)$ is an approximation of $w(x)$, 
we have
\begin{align}
w(x) B_{n}(x; a^{\ast}, \mathcal{D}_{d}) 
=
L_{n}[a^{\ast}; w](x)
\Bigg/
\sum_{k=1}^{n} 
\frac{\lambda_{k}^{\ast}}{S_{d}(x - a_{k}^{\ast})}
\approx 
w(x)
\Bigg/
\sum_{k=1}^{n} 
\frac{\lambda_{k}^{\ast}}{S_{d}(x - a_{k}^{\ast})}.
\label{eq:approx_wB}
\end{align}
Therefore, by replacing the factor $w(x) B_{n}(x; a^{\ast}, \mathcal{D}_{d})$ 
in Formula~\eqref{eq:FirstBary} with the RHS of~\eqref{eq:approx_wB}, 
we obtain the approximate form of the formula as follows: 
\begin{align}
L_{n}[a^{\ast}; f](x) 
& \approx
w(x) 
\sum_{k=1}^{n} 
\frac{\lambda_{k}^{\ast}}{S_{d}(x - a_{k}^{\ast})} \, 
\frac{f(a_{k}^{\ast})}{w(a_{k}^{\ast})}
\Bigg/
\sum_{j=1}^{n} 
\frac{\lambda_{j}^{\ast}}{S_{d}(x - a_{j}^{\ast})}. 
\label{eq:SecondBary}
\end{align}
We denote the RHS of \eqref{eq:SecondBary} by $\tilde{L}_{n}[a^{\ast}; f](x)$. 
Then, we regard this formula as the analogue of the second barycentric formula.

\section{Numerical experiments}
\label{sec:num_algo}

We compute a numerical approximation of the minimizer $a^{\ast}$ of $I_{n}^{\mathrm{D}}$
by Newton's method shown in Figure~\ref{fig:Newton}. 
Recall that $H_{n}(a)$ is the Hessian of $I_{n}^{\mathrm{D}}$ at $a$. 
Let $\tilde{a}^{\ast} = (\tilde{a}_{1}^{\ast}, \ldots , \tilde{a}_{n}^{\ast}) \in \mathcal{R}_{n}$ 
denote the output of this algorithm. 
Then, 
in order to approximate $f(x)$, 
we use the barycentric formulas in~\eqref{eq:FirstBary} and~\eqref{eq:SecondBary}. 
Recall that their explicit forms are given by
\begin{align}
\text{\textbf{(I)}} \qquad 
L_{n}[\tilde{a}^{\ast}; f](x)
& = 
w(x) \left[ \prod_{j = 1}^{n} \tanh \left( \frac{\pi}{4d} (x - \tilde{a}_{j}^{\ast}) \right) \right]
\sum_{k = 1}^{n} \frac{2 \tilde{\lambda}_{k}^{\ast}}{\sinh \left( \frac{\pi}{2d} (x - \tilde{a}_{k}^{\ast}) \right)}
\frac{f(\tilde{a}_{k}^{\ast})}{w(\tilde{a}_{k}^{\ast})}, 
\label{eq:FirstBaryExplicit} \\
\text{\textbf{(II)}} \qquad 
\tilde{L}_{n}[\tilde{a}^{\ast}; f](x)
& = 
w(x)
\sum_{k = 1}^{n} \frac{2 \tilde{\lambda}_{k}^{\ast}}{\sinh \left( \frac{\pi}{2d} (x - \tilde{a}_{k}^{\ast}) \right)}
\frac{f(\tilde{a}_{k}^{\ast})}{w(\tilde{a}_{k}^{\ast})} 
\Bigg/
\sum_{j = 1}^{n} \frac{2 \tilde{\lambda}_{j}^{\ast}}{\sinh \left( \frac{\pi}{2d} (x - \tilde{a}_{j}^{\ast}) \right)},
\label{eq:SecondBaryExplicit} 
\end{align}
where
\begin{align}
\tilde{\lambda}_{k}^{\ast} = \prod_{j \neq k} \frac{1}{\tanh \left( \frac{\pi}{4d} (\tilde{a}_{k}^{\ast} - \tilde{a}_{j}^{\ast}) \right)}
\qquad (k = 1,\ldots , n). 
\end{align}

We used Matlab R2016b programs for the computations presented in this section. 
The sampling points and approximations
were computed by the programs with 
the double precision floating point numbers
and multi-precision numbers with 75 digits, respectively. 
For the multi-precision numbers, we used the Multiprecision Computing Toolbox for Matlab, 
produced by Advanpix (\url{http://www.advanpix.com}, last accessed on December 18, 2017).
These programs used for the computations are available on the web page \cite{bib:Tanaka_MatlabCodes_2018}.

\begin{figure}[ht]
\centering
\framebox{
\begin{minipage}{0.35\linewidth}
\noindent
Initialize $a \in \mathcal{R}_{n}$ \\

\noindent
\textbf{do} \\
\phantom{aaaa} $\delta :=  - (H_{n}(a))^{-1} \, \nabla I_{n}^{\mathrm{D}}(a)$ \\
\phantom{aaaa} $a := a + \delta$ \\
\phantom{aaaa} $\epsilon := \max\{ |\delta_{1}|, \ldots, |\delta_{n}| \}$ \\
\textbf{while} \ $\epsilon \geq 10^{-14}$ \\

\noindent
Output $a$
\end{minipage}
}
\caption{Newton's method for finding the minimizer of $I_{n}^{\mathrm{D}}$}
\label{fig:Newton}
\end{figure}

\subsection{Comparison with the previous formulas in \cite{bib:TaOkaSu_PotApprox_2017}}
\label{sec:VS_old_formulas}

We begin with the comparison of Formula (I) with the previous formula in \cite{bib:TaOkaSu_PotApprox_2017}. 
Because their difference is the method for generating the sampling points, 
we compare the accuracy of the formulas $L_{n}[\tilde{a}^{\ast}; f]$ and $L_{n}[\tilde{a}^{\mathrm{old}}; f]$, 
where $\tilde{a}^{\mathrm{old}} = \{ \tilde{a}_{j}^{\mathrm{old}} \}$ is the sampling points generated by the method in \cite{bib:TaOkaSu_PotApprox_2017}. 
To this end, we choose the functions and weights listed in Table~\ref{tab:func_weight}, 
which are the same as those used in \cite{bib:TaOkaSu_PotApprox_2017}. 
Each weight $w_{i}$ satisfies Assumptions~\ref{assump:w} and~\ref{assump:w_convex} for $d = \pi/4 - \varepsilon$ with $0 < \varepsilon \ll 1$, 
and each function $f_{i}$ satisfies $f_{i} \in \boldsymbol{H}^{\infty}(\mathcal{D}_{\pi/4- \varepsilon}, w_{i})$ for the corresponding weight $w_{i}$. 
In the following, we adopt $\varepsilon = 10^{-10}$. 

For computing the errors of the formulas, 
we choose $1001$ evaluation points $x_{\ell} \subset \mathbf{R}$ 
and adopt the value $\max_{\ell}| f(x_{\ell}) - (\text{the value of the approximant at } x_{\ell}) |$ as the error. 
First we find a value of $x_{1} \leq 0$ satisfying that
\begin{align}
w_{i}(x_{1}) \leq 
\begin{cases}
10^{-20} & (i = 1), \\
10^{-30} & (i = 2), \\
10^{-75} & (i = 3), 
\end{cases}
\notag 
\end{align}
and then determine the points $x_{\ell}$ by $x_{1001} = -x_{1}$ and 
\begin{align}
\quad x_{\ell} = x_{1} + \frac{x_{1001} - x_{1}}{1000} (\ell-1) \quad (\ell = 2,\ldots, 1000). 
\label{eq:eval_points}
\end{align}
We adopt $x_{1} = -25, -10$ and $-3$ for the computations of $f_{1}$, $f_{2}$ and $f_{3}$, respectively. 

First, 
we present the computed sampling points $\tilde{a}^{\ast}$ and 
functions $U^{\mathrm{D}}(\tilde{a}^{\ast}; x) + Q(x)$ in Figures~\ref{fig:samp_SE}--\ref{fig:samp_DE}. 
Next, 
we present the computed errors in Figures~\ref{fig:func_SE_with_New}--\ref{fig:func_DE_with_New}. 
From these results, 
we can observe that the computed sampling points are very close to those by the previous method 
and the proposed formulas are competitive with the previous formulas.

\begin{rem}
\label{rem:ex_func_on_entire_R}
The functions $f_{1}$ and $f_{3}$ in Table~\ref{tab:func_weight} 
can be derived from the function
\begin{align}
g_{0}(t) = \frac{1}{\sqrt{1 + t^{2}}}
\label{eq:ex_func_on_entire_R}
\end{align}
by the variable transformations 
\begin{align*}
t = \sinh(2x) 
\qquad \text{and} \qquad
t = \sinh \left( \frac{\pi}{2} \sinh (2x) \right), 
\end{align*}
respectively. 
Therefore, 
the approximations of $f_{1}$ and $f_{3}$ by the proposed method
can be regarded as the approximations of $g_{0}$ in \eqref{eq:ex_func_on_entire_R} 
on the entire real line $\mathbf{R}$
via these transformations. 
Thus, 
the proposed method can provide formulas for approximating functions 
such as $g_{0}$ with algebraic decay on $\mathbf{R}$, 
which are discussed in 
\ifthenelse{\value{journal} = 0}{%
\cite[\S 17.9]{bib:Boyd_spec_book2_2000}.
}{
\citet[\S 17.9]{bib:Boyd_spec_book2_2000}.
}
\end{rem} 

\begin{table}[H]
\begin{center}
\caption{The functions and weights used in {\cite{bib:TaOkaSu_PotApprox_2017}}.}
\label{tab:func_weight}
\begin{tabular}{l | l}
$f_{i} \in \boldsymbol{H}^{\infty}(\mathcal{D}_{\pi/4- \varepsilon}, w_{i})$ & $w_{i}$ \\
\hline
1.\quad $f_{1}(x) = \mathop{\mathrm{sech}}(2x)$ & $w_{1}(x) = \mathop{\mathrm{sech}}(2x)$ \\
2.\quad $f_{2}(x) = \dfrac{x^2}{(\pi/4)^2 + x^2}\, \mathrm{e}^{-x^{2}} $ & $w_{2}(x) = \mathrm{e}^{-x^{2}}$ \\
3.\quad $f_{3}(x) = \mathop{\mathrm{sech}}((\pi/2) \sinh(2x))$ & $w_{3}(x) = \mathop{\mathrm{sech}}((\pi/2) \sinh(2x))$  
\end{tabular}
\end{center}
\end{table}

\begin{figure}[ht]
\noindent
\begin{center}
\begin{minipage}[t]{0.49\linewidth}
\begin{center}
\includegraphics[width=\linewidth]{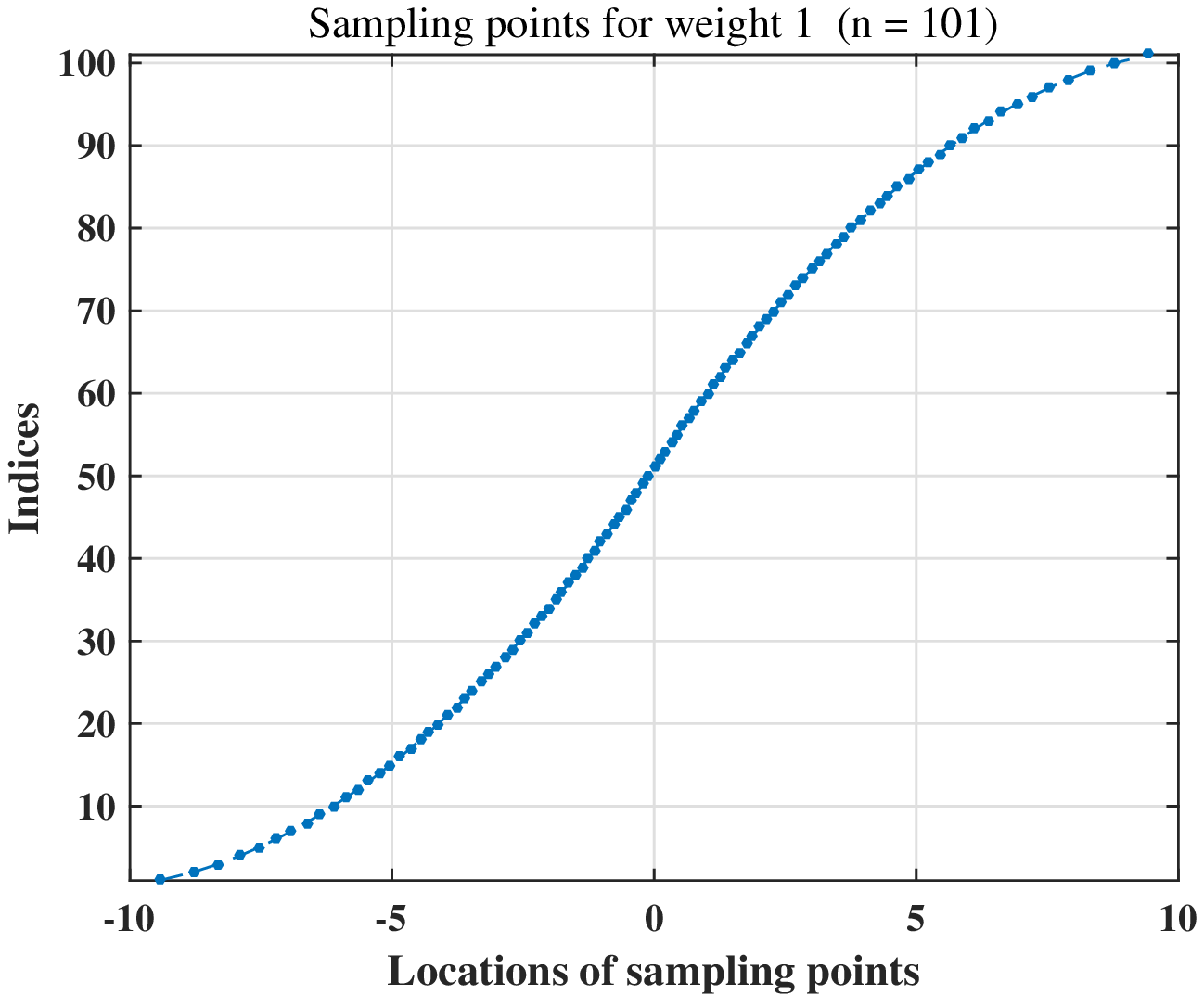}
(a) Sampling points $\{ \tilde{a}_{j}^{\ast} \}$
\end{center}
\end{minipage}
\end{center}
\noindent
\begin{minipage}[t]{0.49\linewidth}
\begin{center}
\includegraphics[width=\linewidth]{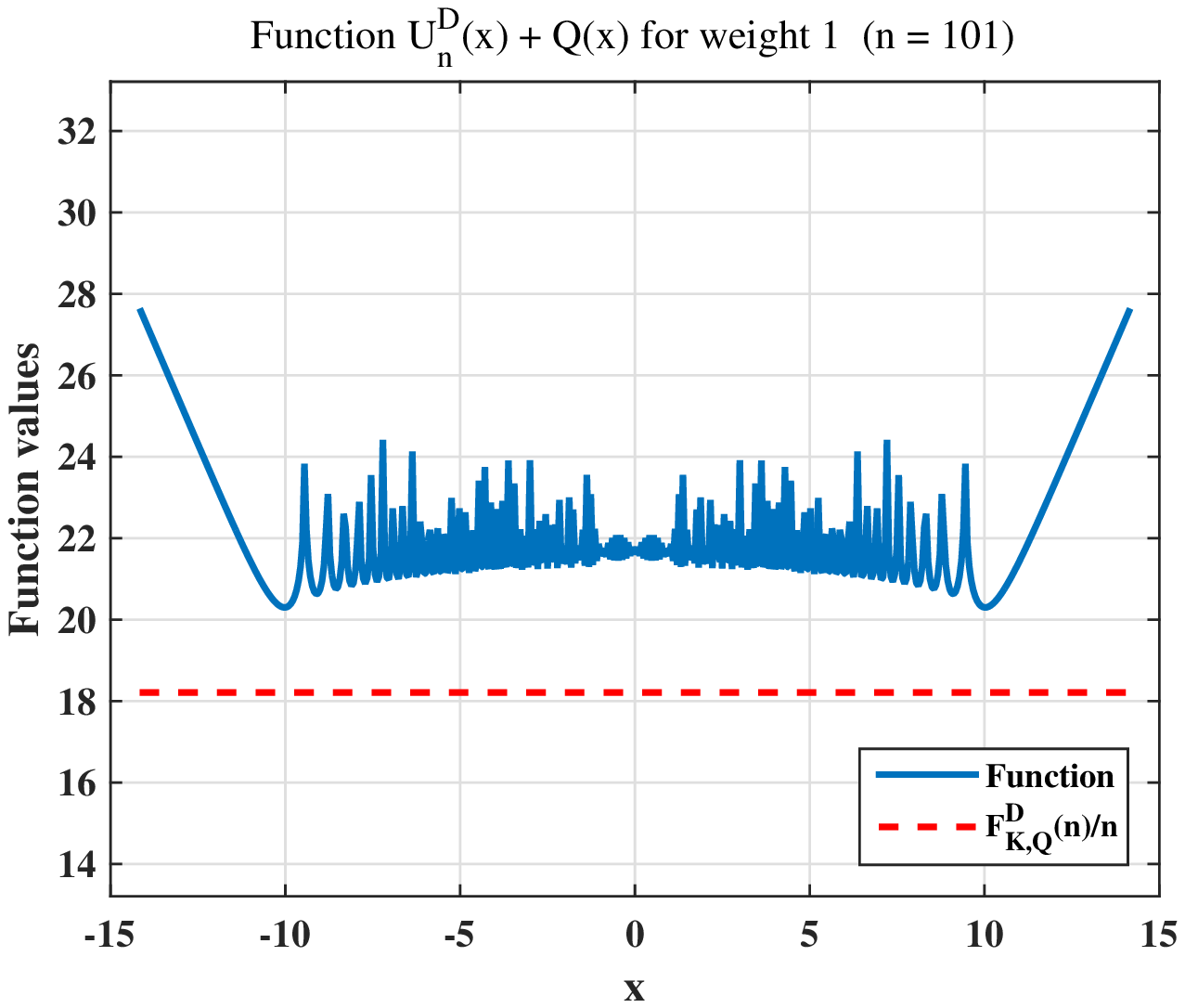}
(b) Function $U_{n}^{\mathrm{D}}(\tilde{a}^{\ast};x) + Q(x)$
\end{center}
\end{minipage}
\begin{minipage}[t]{0.49\linewidth}
\begin{center}
\includegraphics[width=\linewidth]{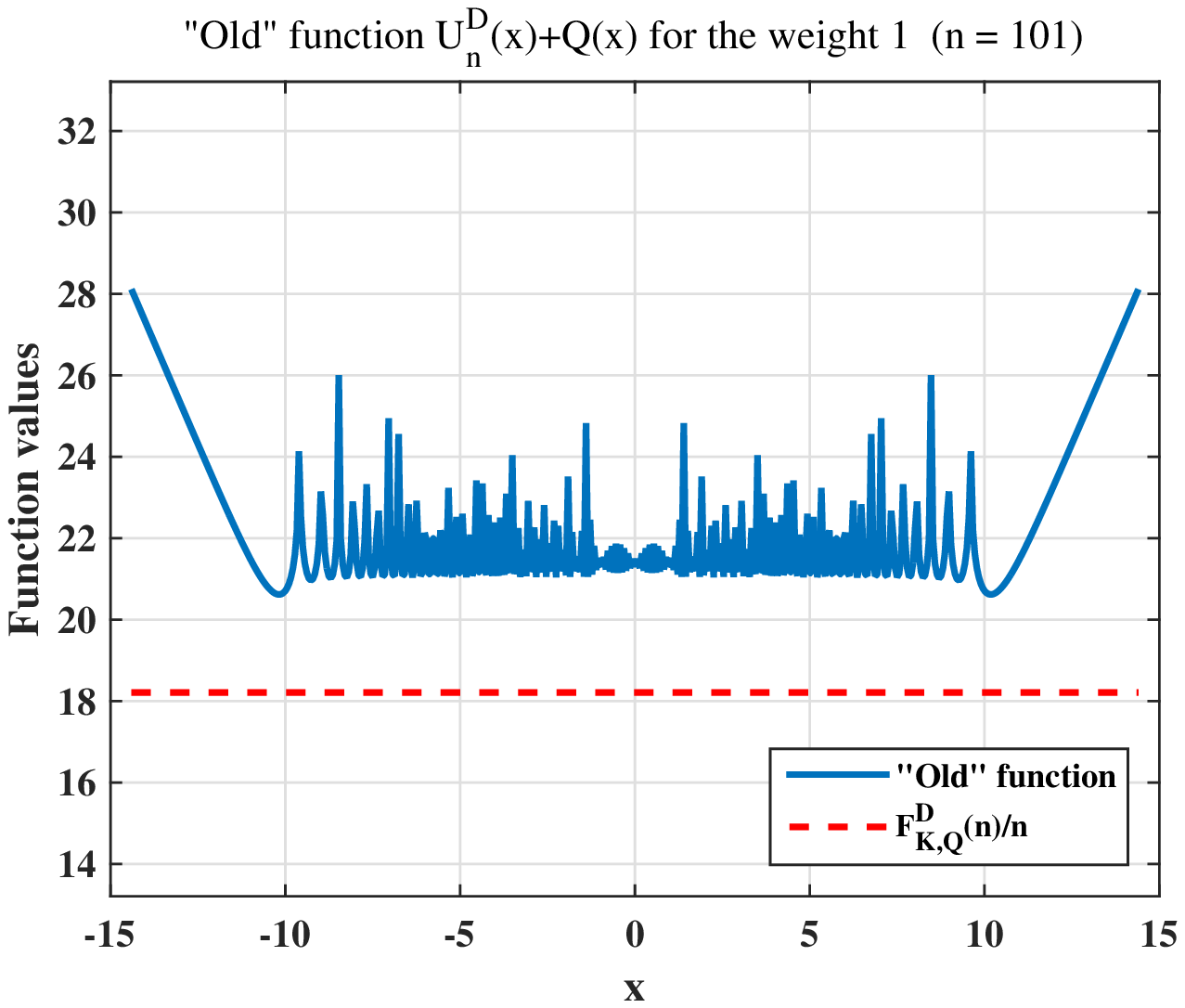}
(c) Function $U_{n}^{\mathrm{D}}(\tilde{b}^{\ast};x) + Q(x)$ 
for the sampling points $\{ \tilde{b}_{j}^{\ast} \}$ by the method in \cite{bib:TaOkaSu_PotApprox_2017}
\end{center}
\end{minipage}
\bigskip
\caption{Results for the sampling points for weight 1 ($w_{1}$) in Table~\ref{tab:func_weight} and $n = 101$.}
\label{fig:samp_SE}
\end{figure}

\begin{figure}[ht]
\noindent
\begin{center}
\begin{minipage}[t]{0.49\linewidth}
\begin{center}
\includegraphics[width=\linewidth]{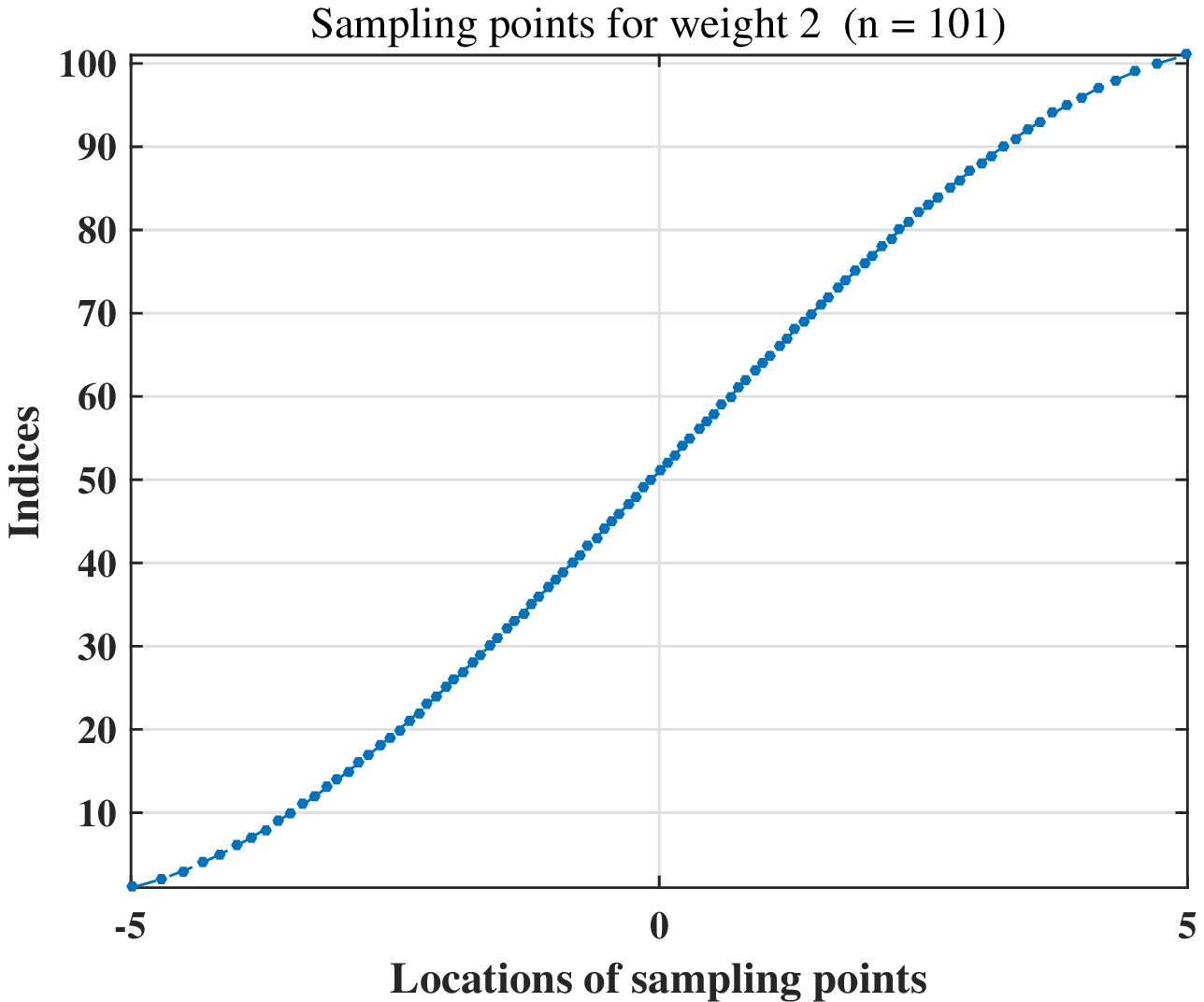}
(a) Sampling points $\{ \tilde{a}_{j}^{\ast} \}$
\end{center}
\end{minipage}
\end{center}
\noindent
\begin{minipage}[t]{0.49\linewidth}
\begin{center}
\includegraphics[width=\linewidth]{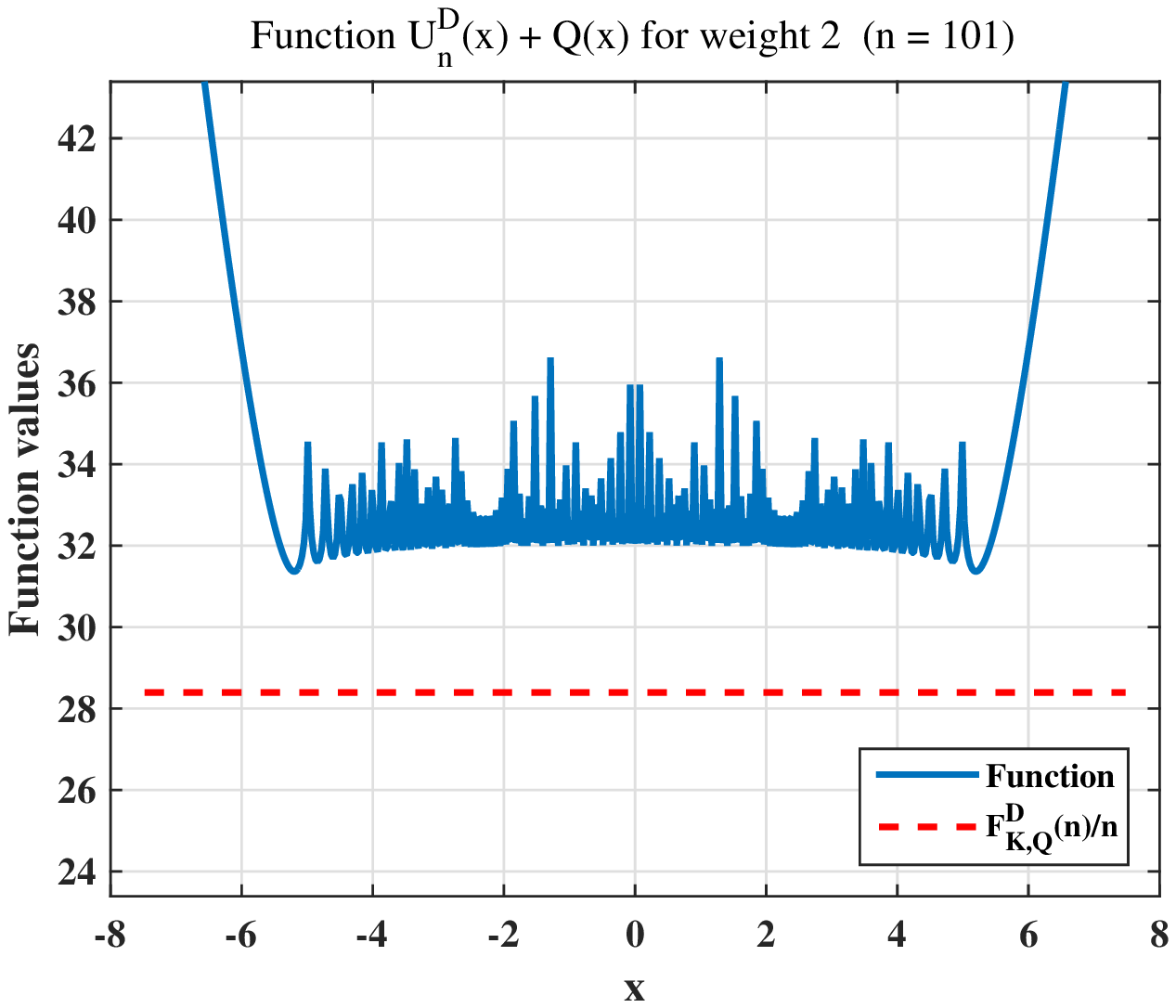}
(b) Function $U_{n}^{\mathrm{D}}(\tilde{a}^{\ast};x) + Q(x)$. 
\end{center}
\end{minipage}
\begin{minipage}[t]{0.49\linewidth}
\begin{center}
\includegraphics[width=\linewidth]{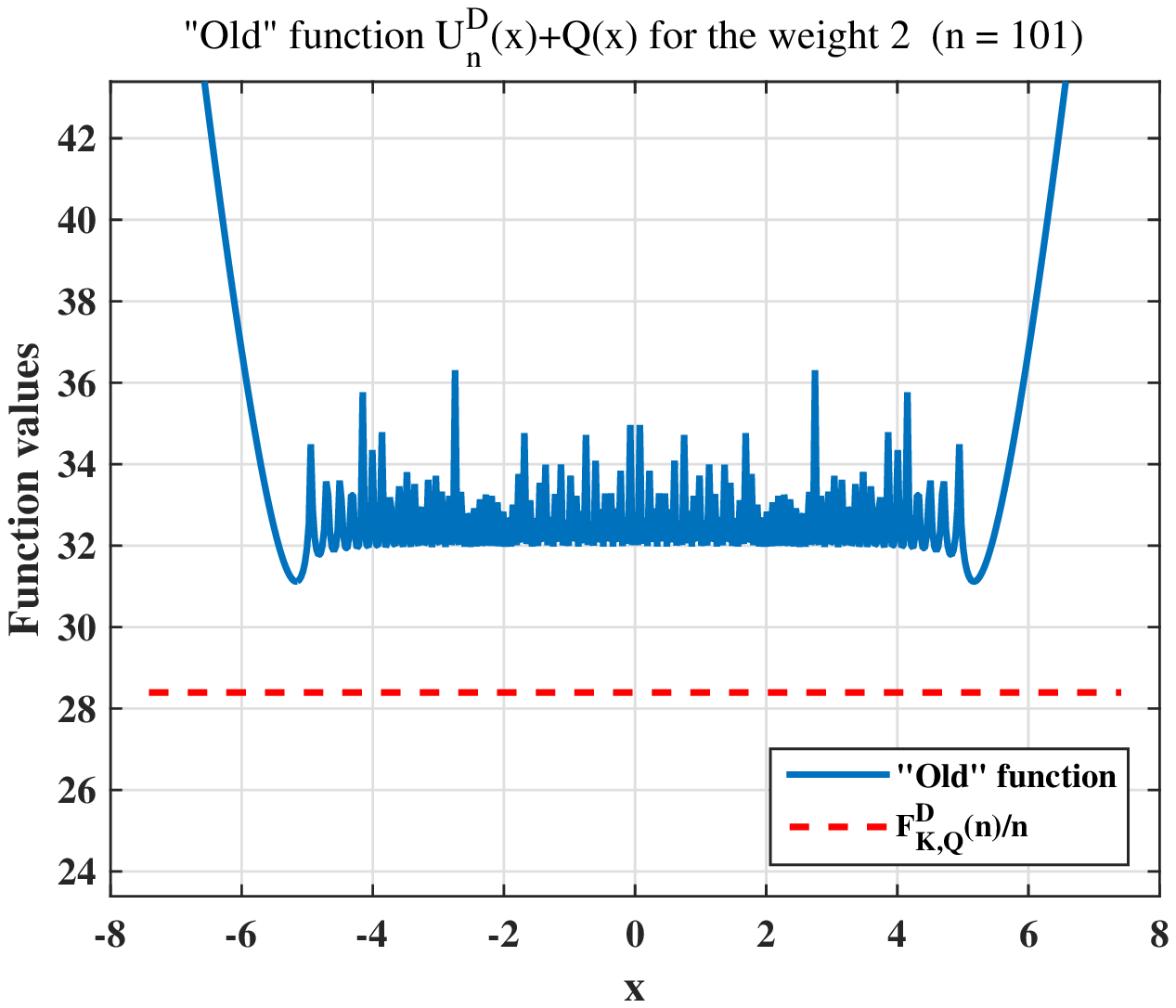}
(c) Function $U_{n}^{\mathrm{D}}(\tilde{b}^{\ast};x) + Q(x)$ 
for the sampling points $\{ \tilde{b}_{j}^{\ast} \}$ by the method in \cite{bib:TaOkaSu_PotApprox_2017}
\end{center}
\end{minipage}
\bigskip
\caption{Results for the sampling points for weight 2 ($w_{2}$) in Table~\ref{tab:func_weight} and $n = 101$.}
\label{fig:samp_Ga}
\end{figure}

\begin{figure}[ht]
\noindent
\begin{center}
\begin{minipage}[t]{0.49\linewidth}
\begin{center}
\includegraphics[width=\linewidth]{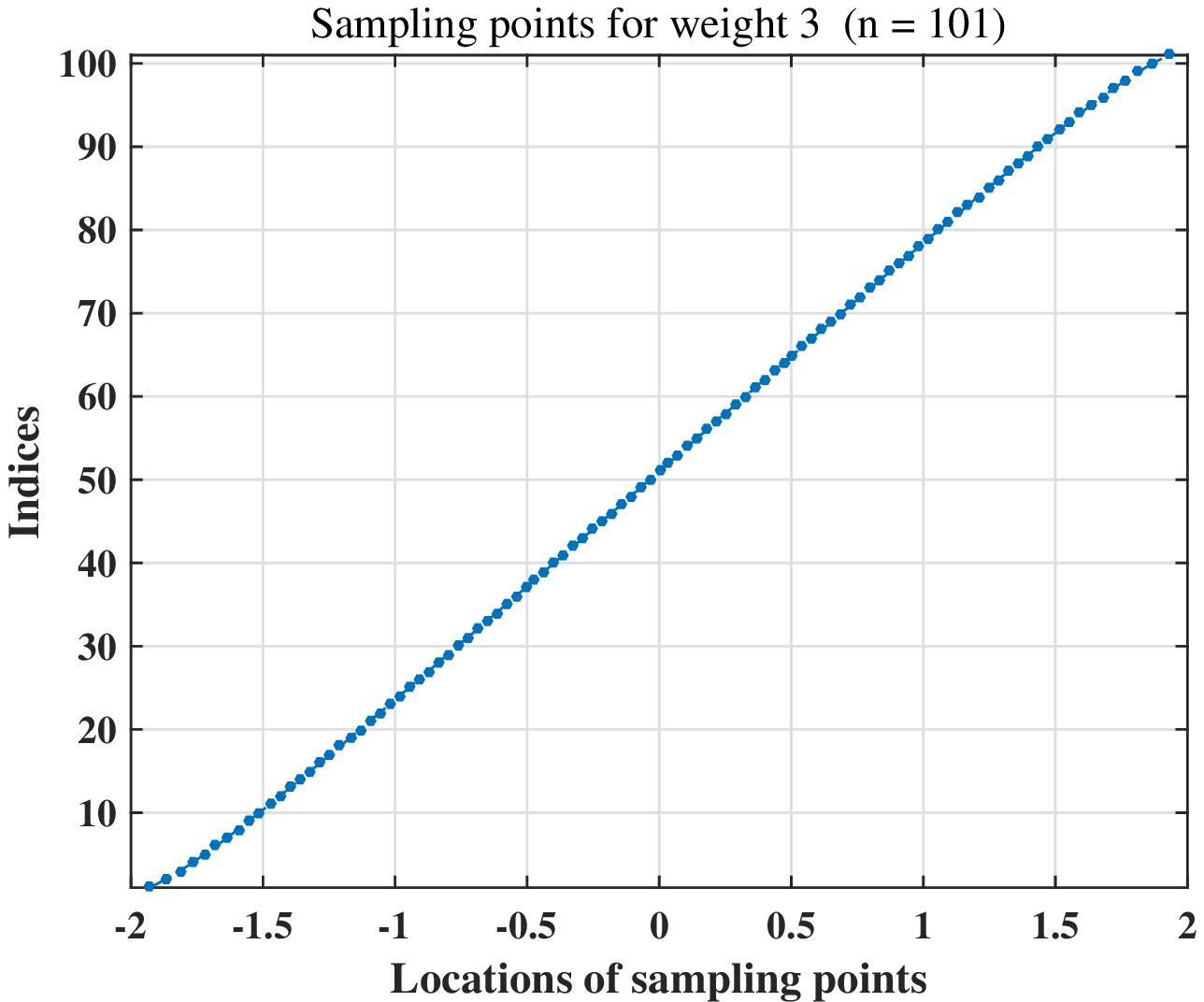}
(a) Sampling points $\{ \tilde{a}_{j}^{\ast} \}$
\end{center}
\end{minipage}
\end{center}
\noindent
\begin{minipage}[t]{0.49\linewidth}
\begin{center}
\includegraphics[width=\linewidth]{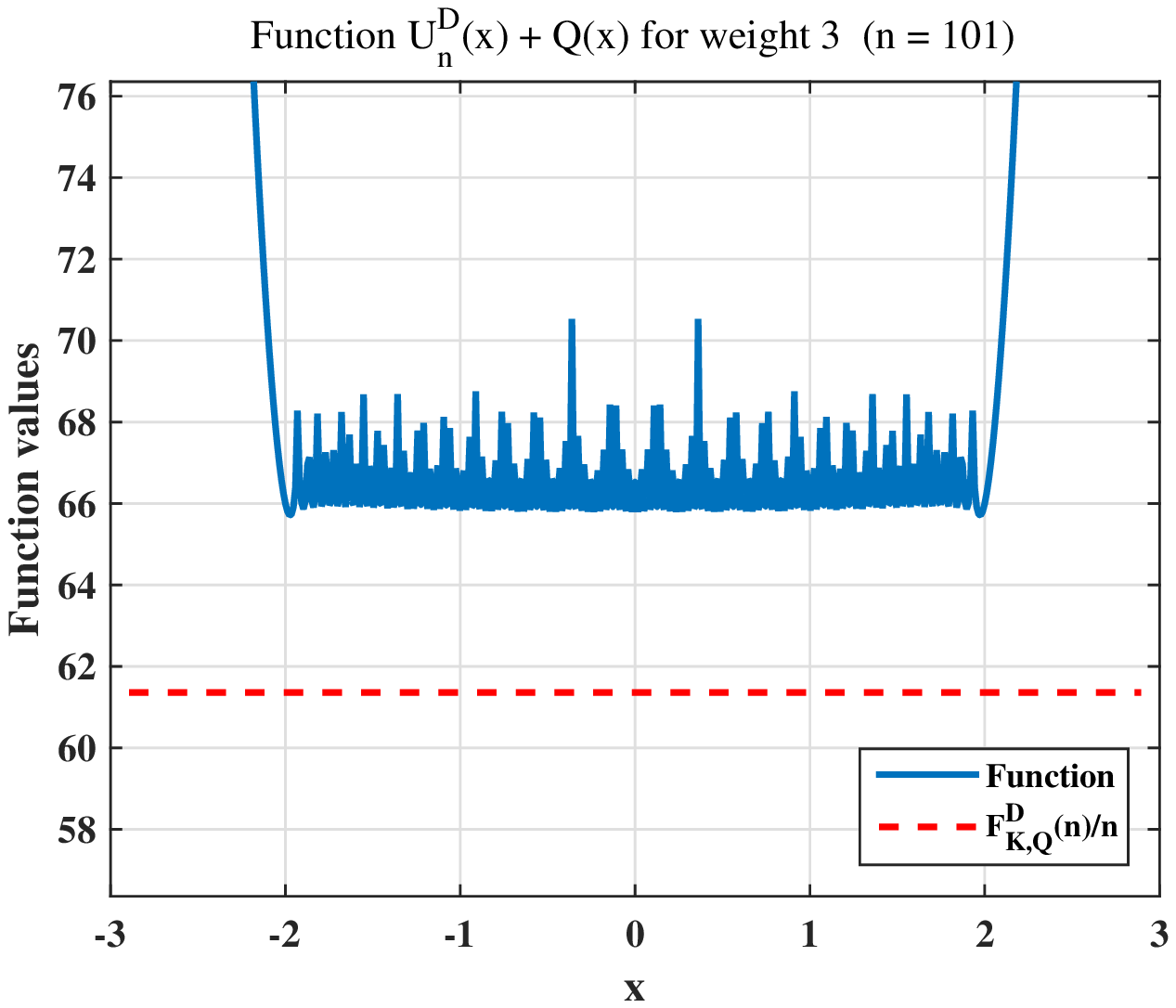}
(b) Function $U_{n}^{\mathrm{D}}(\tilde{a}^{\ast};x) + Q(x)$
\end{center}
\end{minipage}
\begin{minipage}[t]{0.49\linewidth}
\begin{center}
\includegraphics[width=\linewidth]{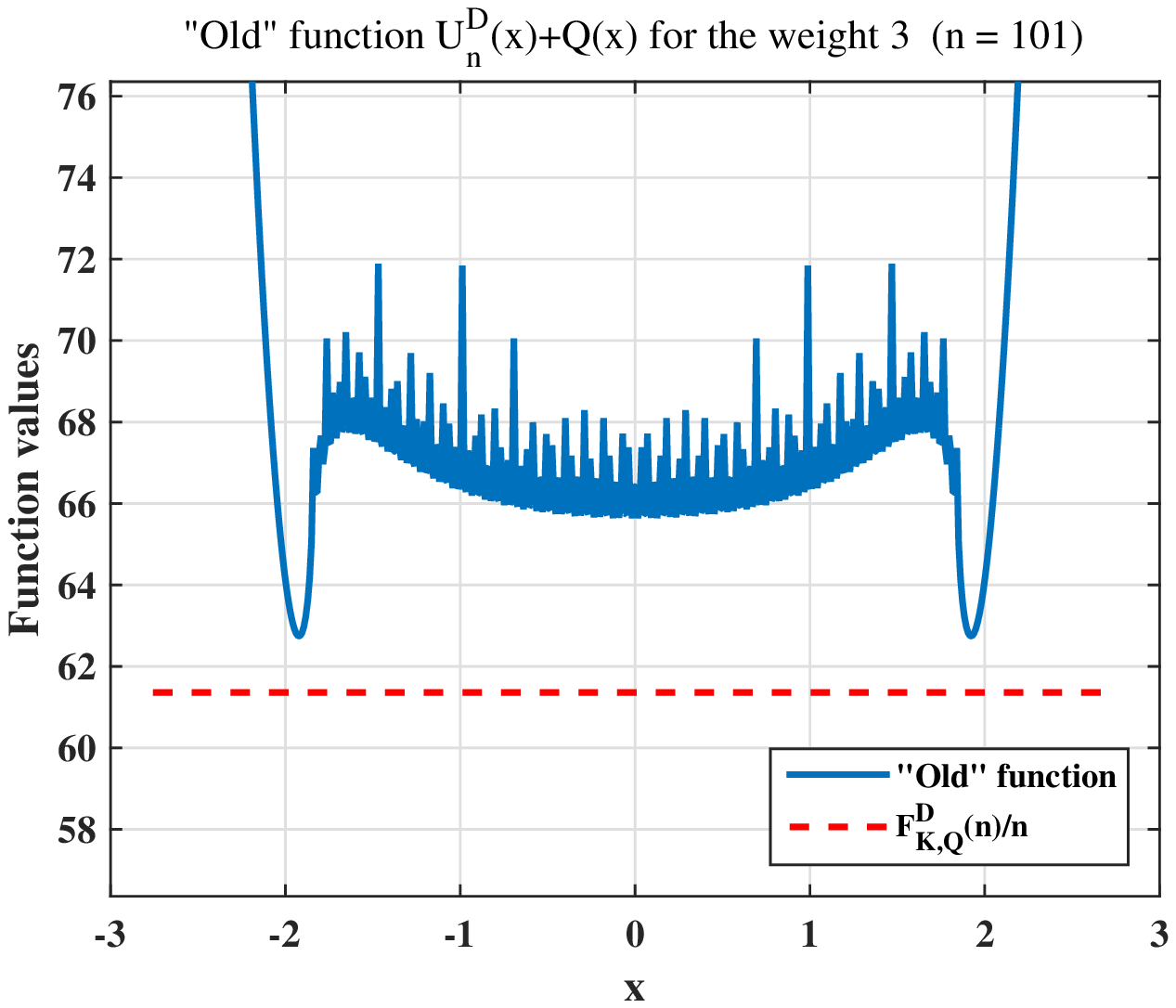}
(c) Function $U_{n}^{\mathrm{D}}(\tilde{b}^{\ast};x) + Q(x)$ 
for the sampling points $\{ \tilde{b}_{j}^{\ast} \}$ by the method in \cite{bib:TaOkaSu_PotApprox_2017}
\end{center}
\end{minipage}
\bigskip
\caption{Results for the sampling points for weight 3 ($w_{3}$) in Table~\ref{tab:func_weight} and $n = 101$.}
\label{fig:samp_DE}
\end{figure}

\begin{figure}[ht]
\begin{minipage}[t]{0.48\linewidth}
\begin{center}
\includegraphics[width=\linewidth]{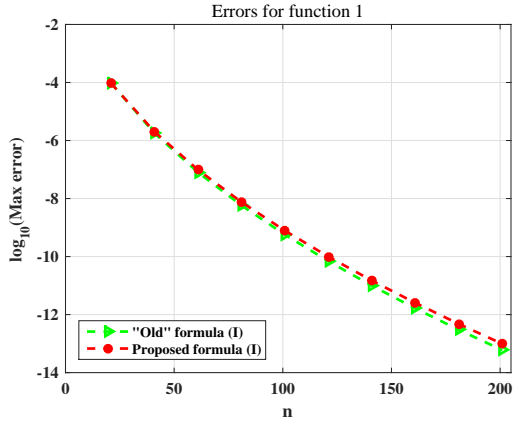}
\end{center}
\caption{Errors for function 1 ($f_{1}$) in Table~\ref{tab:func_weight} . ``Old'' formula refers to that in \cite{bib:TaOkaSu_PotApprox_2017}.}
\label{fig:func_SE_with_New}
\end{minipage}
\quad
\begin{minipage}[t]{0.48\linewidth}
\begin{center}
\includegraphics[width=\linewidth]{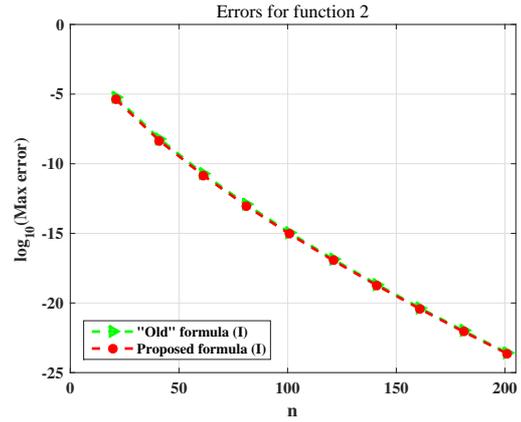}
\end{center}
\caption{Errors for function 2 ($f_{2}$) in Table~\ref{tab:func_weight}. ``Old'' formula refers to that in \cite{bib:TaOkaSu_PotApprox_2017}.}
\label{fig:func_Ga_with_New}
\end{minipage}
\end{figure}

\begin{figure}[ht]
\begin{center}
\includegraphics[width=0.5\linewidth]{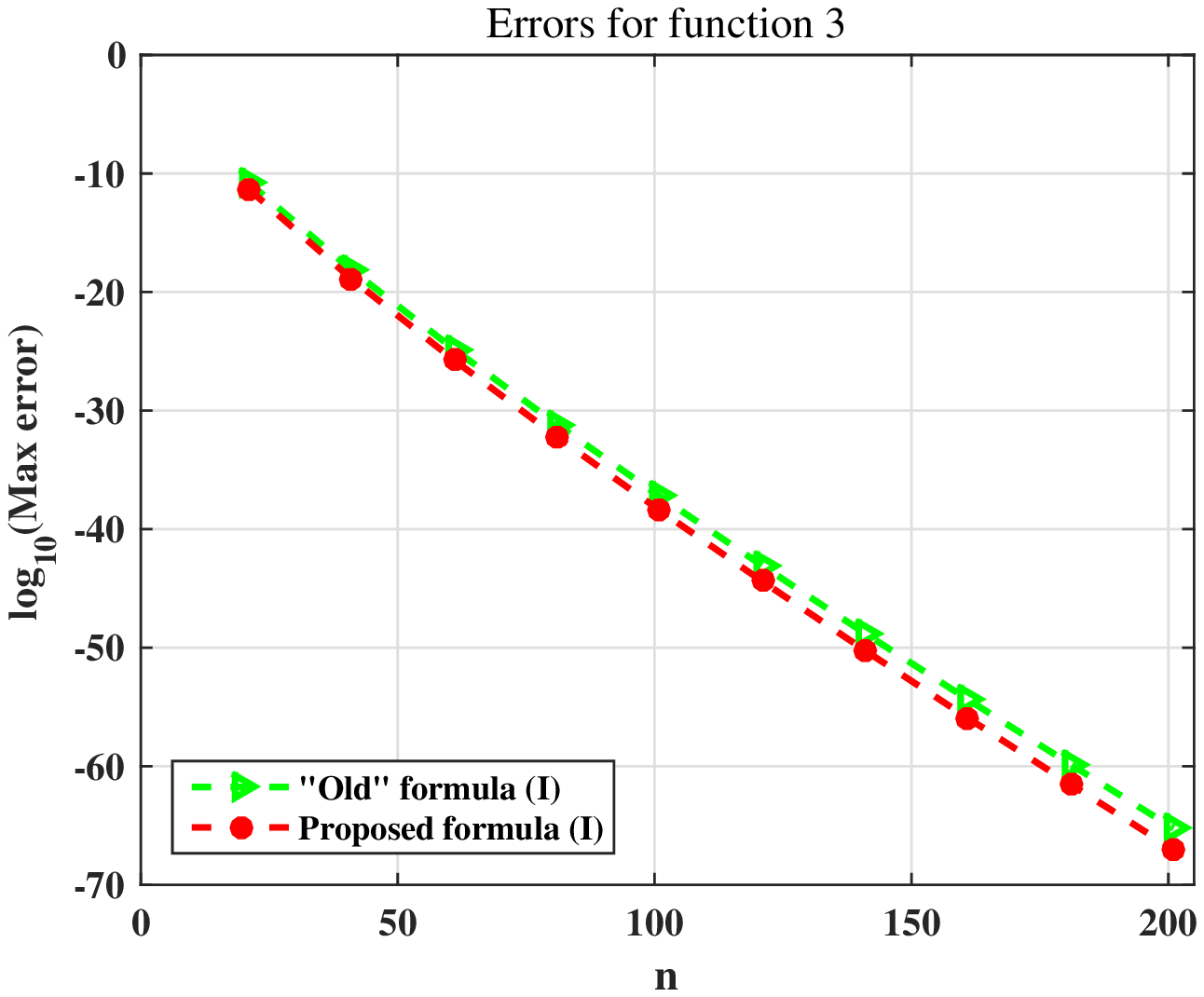}
\end{center}
\caption{Errors for function 3 ($f_{3}$) in Table~\ref{tab:func_weight}. ``Old'' formula refers to that in \cite{bib:TaOkaSu_PotApprox_2017}.}
\label{fig:func_DE_with_New}
\end{figure}

\subsection{Comparison with the sinc interpolations with transformations}
\label{sec:VS_Sinc}

Next, we compare Formulas (I), (II) and the sinc interpolation with a transformation. 
We consider the function $g_{1}$ given by
\begin{align}
g_{1}(t) =  \sqrt{1 - t^{2}} \, (1+t^{2})
\end{align}
and its approximation for $t \in (-1,1)$. 
The function $g_{1}$ has the singularities at the endpoints $t = \pm 1$. 
In order to mitigate the difficulty in approximating $g$ at their neighborhoods, 
we employ the variable transformations given by $\psi_{1}$ and $\psi_{2}$ 
in \eqref{eq:TANHtrans} and~\eqref{eq:DEtrans}, respectively. 
Then, we consider the approximations of the transformed functions
\begin{align}
& f_{4}(x) = g_{1}(\psi_{1}(x)) = w_{4}(x) \left( 1 + \tanh \left( \frac{x}{2} \right)^{2} \right), \label{eq:f_TANHtrans} \\
& f_{5}(x) = g_{1}(\psi_{2}(x)) = w_{5}(x) \left( 1 + \tanh \left( \frac{\pi}{2} \sinh x \right)^{2} \right) \label{eq:f_DEtrans}
\end{align}
for $x \in \mathbf{R}$, where
\begin{align}
& w_{4}(x) = \mathop{\mathrm{sech}} \left( \frac{x}{2} \right),  \label{eq:w_TANHtrans} \\
& w_{5}(x) = \mathop{\mathrm{sech}} \left( \frac{\pi}{2} \sinh x \right). \label{eq:w_DEtrans} 
\end{align}
By letting 
\begin{align}
d_{4} = \pi - \varepsilon, 
\qquad d_{5} = \frac{\pi}{2} - \varepsilon
\label{eq:d_i}
\end{align}
with $0 < \varepsilon \ll 1$, for $i = 4,5$, 
we can confirm that the weight function $w_{i}$ satisfies
Assumptions~\ref{assump:w} and~\ref{assump:w_convex} for $d = d_{i}$. 
Furthermore, the assertion $f_{i} \in \boldsymbol{H}^{\infty}(\mathcal{D}_{d_{i}}, w_{i})$ holds true for $i = 4,5$. 
In the following, we adopt $\varepsilon = 10^{-10}$. 

For the functions $f_{4}$ and $f_{5}$, 
we compare Formulas (I), (II), and the sinc interpolation formula~\eqref{eq:trans_sinc_approx}:
\begin{align}
f(x) \approx \sum_{k = -N_{-}}^{N_{+}} f(kh) \, \mathop{\mathrm{sinc}}(x/h - k). 
\label{eq:trans_sinc_approx_revisited}
\end{align}
We need to determine the parameters $N_{\pm}$ and $h$ in this formula. 
Since the weights $w_{i}$ are even, 
we consider odd $n$ as the numbers of the sampling points 
and set $N_{-} = N_{+} = (n-1)/2$. 
Furthermore, we adopt the width $h > 0$ 
so that the orders of the sampling error and truncation error (almost) coincide
\ifthenelse{\value{journal} = 0}{%
\cite{bib:Sugihara_NearOpt_2003, bib:TaSuMu_DE_Sinc_2009}. 
}{
\citep{bib:Sugihara_NearOpt_2003, bib:TaSuMu_DE_Sinc_2009}. 
}
Actually, 
the former is $\mathrm{O}(\exp(-\pi d_{i}/h)) \ (h \to 0)$ and 
the latter is estimated depending on the weight $w_{i}$ as follows:
\begin{align}
\sum_{|k| > (n-1)/2} | f_{i}(kh) |
=
\begin{cases}
\mathrm{O}( \exp(-nh/4) ) & (i = 4), \\
\mathrm{O}( \exp(-(\pi/4) \exp(nh/2)) ) & (i = 5),
\end{cases}
\qquad (n \to \infty).
\notag
\end{align}
Then, we set
\begin{align}
h = 
\begin{cases}
(4 \pi (\pi - \varepsilon) / n)^{1/2} &  (i = 4), \\
2 n^{-1} \log ((\pi - 2\varepsilon) n) & (i = 5).
\end{cases}
\notag
\end{align}

We choose the evaluation points $x_{\ell}$ in a similar manner to that of Section~\ref{sec:VS_old_formulas}. 
We find a value of $x_{1} \leq 0$ satisfying that
\begin{align}
w_{i}(x_{1}) \leq 
\begin{cases}
10^{-20} & (i = 4), \\
10^{-75} & (i = 5), 
\end{cases}
\notag 
\end{align}
and determine the points $x_{\ell}$ by $x_{1001} = -x_{1}$ and~\eqref{eq:eval_points}. 
We adopt $x_{1} = -100$ and $-6$ for the computations of $f_{4}$ and $f_{5}$, respectively. 

We show the errors of Formulas (I), (II) and sinc formula 
for $f_{4}$ and $f_{5}$ in Figures~\ref{fig:func_w4_even_se} and~\ref{fig:func_w5_even_de}, respectively. 
We can observe that Formulas (I) and (II) have approximately the same accuracy
and they outperform the sinc formula for each case.  

\begin{figure}[ht]
\begin{minipage}[t]{0.48\linewidth}
\begin{center}
\includegraphics[width=\linewidth]{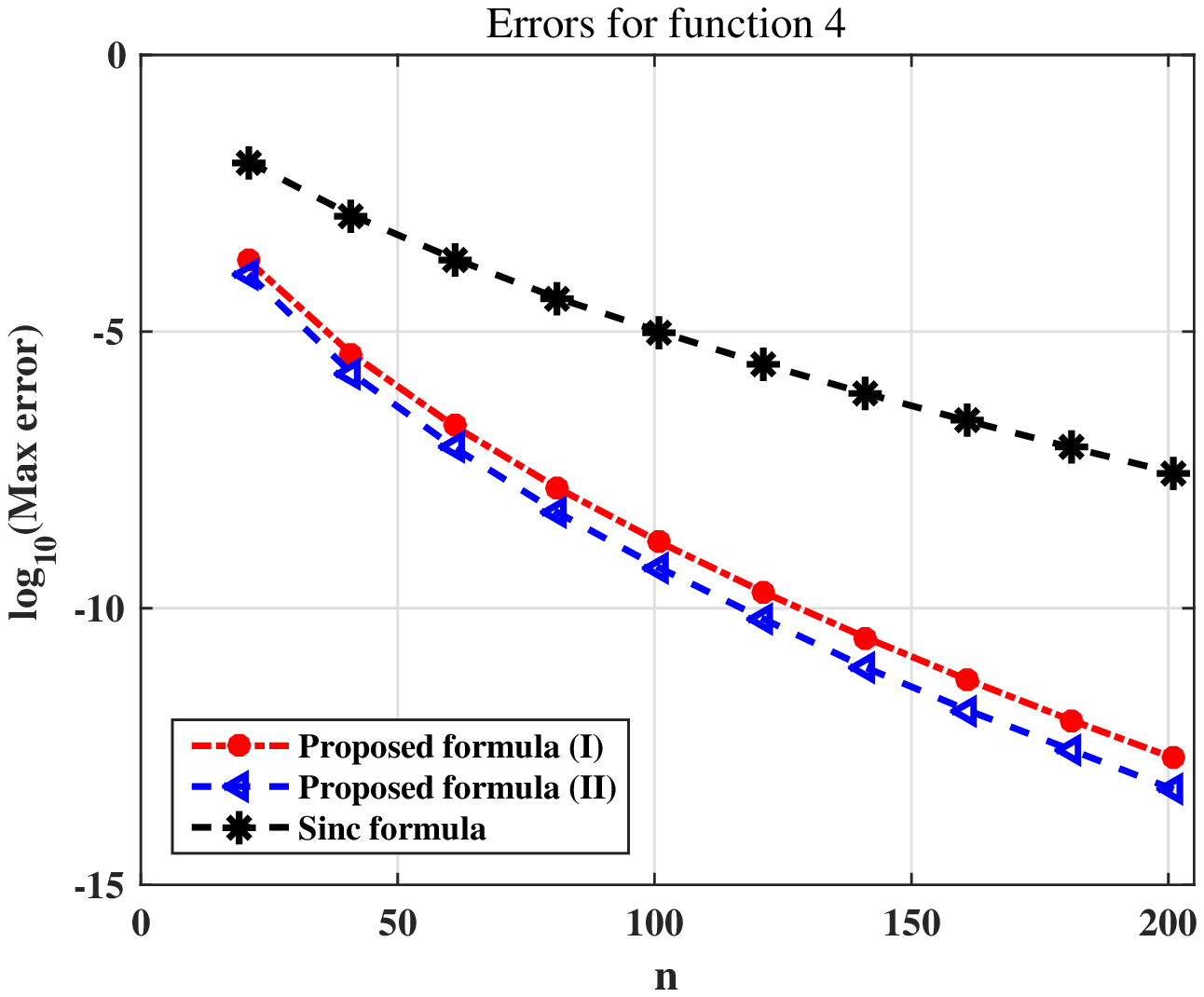}
\end{center}
\caption{Errors for function 4 ($f_{4}$). }
\label{fig:func_w4_even_se}
\end{minipage}
\quad
\begin{minipage}[t]{0.48\linewidth}
\begin{center}
\includegraphics[width=\linewidth]{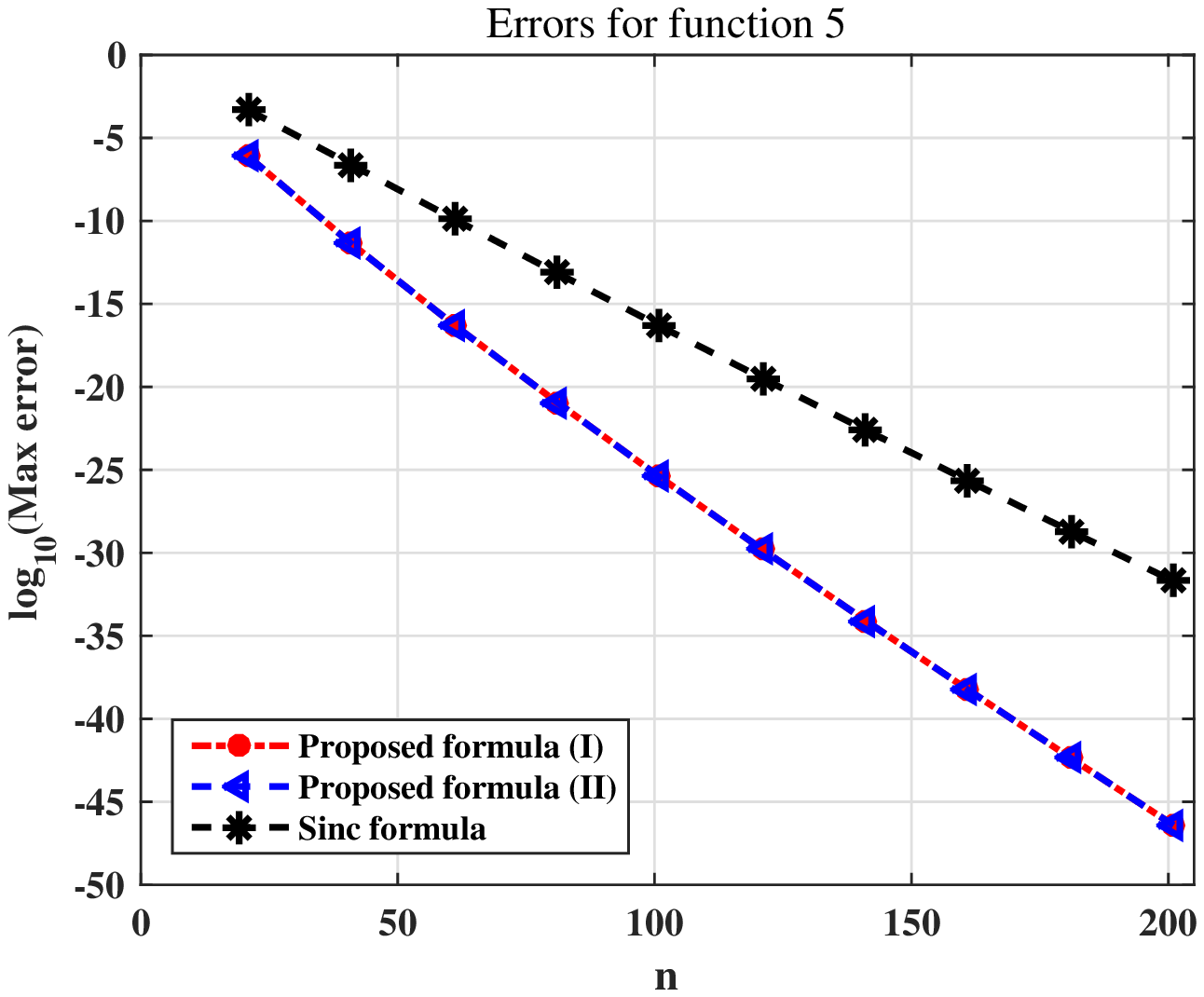}
\end{center}
\caption{Errors for function 5 ($f_{5}$). }
\label{fig:func_w5_even_de}
\end{minipage}
\end{figure}

\subsection{Uneven weight functions}
\label{sec:Uneven}

Finally, we approximate functions with uneven weights, to which the method in \cite{bib:TaOkaSu_PotApprox_2017} cannot be applied. 
To this end, we consider the function $g_{2}$ given by 
\begin{align}
g_{2}(t) = (1-t)^{1/2} (1+t)^{3/2} (1+t^{2})
\end{align}
and its approximation for $t \in (-1,1)$. 
In a similar manner to that in Section~\ref{sec:VS_Sinc}, 
we consider the transformed functions
\begin{align}
& f_{6}(x) = g_{2}(\psi_{1}(x)) = w_{6}(x) \cdot 4 \left( 1 + \tanh \left( \frac{x}{2} \right)^{2} \right), \label{eq:un_f_TANHtrans} \\
& f_{7}(x) = g_{2}(\psi_{2}(x)) = w_{7}(x) \cdot 4 \left( 1 + \tanh \left( \frac{\pi}{2} \sinh x \right)^{2} \right) \label{eq:un_f_DEtrans}
\end{align}
for $x \in \mathbf{R}$, where
\begin{align}
& w_{6}(x) = \frac{1}{(1+\mathrm{e}^{x})^{1/2}(1+\mathrm{e}^{-x})^{3/2}}, \\
& w_{7}(x) = \frac{1}{(1+\mathrm{e}^{\pi \sinh x})^{1/2}(1+\mathrm{e}^{-\pi \sinh x})^{3/2}}. 
\end{align}
By letting 
\begin{align}
d_{6} = \pi - \varepsilon, 
\qquad d_{7} = \frac{\pi}{2} - \varepsilon
\label{eq:d_i_2}
\end{align}
with $0 < \varepsilon \ll 1$, for $i = 6,7$, 
we can confirm that the weight function $w_{i}$ satisfies
Assumptions~\ref{assump:w} and~\ref{assump:w_convex} for $d = d_{i}$. 
Furthermore, the assertion $f_{i} \in \boldsymbol{H}^{\infty}(\mathcal{D}_{d_{i}}, w_{i})$ holds true for $i = 6,7$. 
In the following, we adopt $\varepsilon = 10^{-10}$. 

For the functions $f_{6}$ and $f_{7}$, 
we also compare Formulas (I), (II) and the sinc interpolation formula in~\eqref{eq:trans_sinc_approx_revisited}. 
In these cases, we need to take the unevenness of the weights into account 
in determining the parameters $N_{\pm}$ and $h$ in~\eqref{eq:trans_sinc_approx_revisited}. 
Since 
\begin{align*}
& w_{6}(x) = 
\begin{cases}
\mathrm{O}(\mathrm{e}^{(3/2)x}) & (x \to -\infty), \\
\mathrm{O}(\mathrm{e}^{-(1/2)x}) & (x \to +\infty), 
\end{cases} \\
& w_{7}(x) = 
\begin{cases}
\mathrm{O}(\exp(- (3\pi/4) \, \mathrm{e}^{-x})) & (x \to -\infty), \\
\mathrm{O}(\exp(- (\pi/2) \, \mathrm{e}^{x})) & (x \to +\infty),
\end{cases}
\end{align*}
we adopt 
\begin{align*}
\begin{array}{llll}
h = \sqrt{\dfrac{8 \pi d_{3}}{3 n}}, & N_{-} = \left \lfloor \dfrac{1}{4} n \right \rfloor, & N_{+} = - N_{-} + n-1 &  \quad \text{for } w_{6},  \\
\phantom{\tiny a} & & & \\
h = \dfrac{2}{n} \log \dfrac{d_{4} n}{\sqrt{3/2}}, & N_{-} = \left \lfloor \dfrac{n}{2} - \dfrac{1}{2h} \log \dfrac{3}{2} \right \rfloor, & N_{+} = - N_{-} + n-1 & \quad \text{for } w_{7}. 
\end{array}
\end{align*}

We choose the evaluation points $x_{\ell}$ in a similar manner to that of Section~\ref{sec:VS_old_formulas}. 
We find values of $x_{1} \leq 0$ and $x_{1001} \geq 0$ satisfying that
\begin{align}
w_{i}(x_{1}), \ w_{i}(x_{1001}) \leq 
\begin{cases}
10^{-20} & (i = 6), \\
10^{-75} & (i = 7), 
\end{cases}
\notag 
\end{align}
and determine the points $x_{\ell}$ by~\eqref{eq:eval_points}. 
We adopt $(x_{1}, x_{1001}) = (-40, 100)$ and $(-4.5, 5.5)$ for the computations of $f_{6}$ and $f_{7}$, respectively. 

We show the sampling points and functions $U_{n}^{\mathrm{D}}(\tilde{a}^{\ast}; x) + Q(x)$ 
for these cases in Figures~\ref{fig:sample_w6_uneven_de}--\ref{fig:pot_w7_uneven_de}. 
Furthermore, we show the errors of Formulas (I), (II) and the sinc formula 
for $f_{6}$ and $f_{7}$ in Figures~\ref{fig:func_w6_uneven_se} and~\ref{fig:func_w7_uneven_de}, respectively. 
We can observe that Formulas (I) and (II) have approximately the same accuracy
and they outperform the sinc formula for each case.  

\begin{figure}[ht]
\begin{minipage}[t]{0.48\linewidth}
\begin{center}
\includegraphics[width=\linewidth]{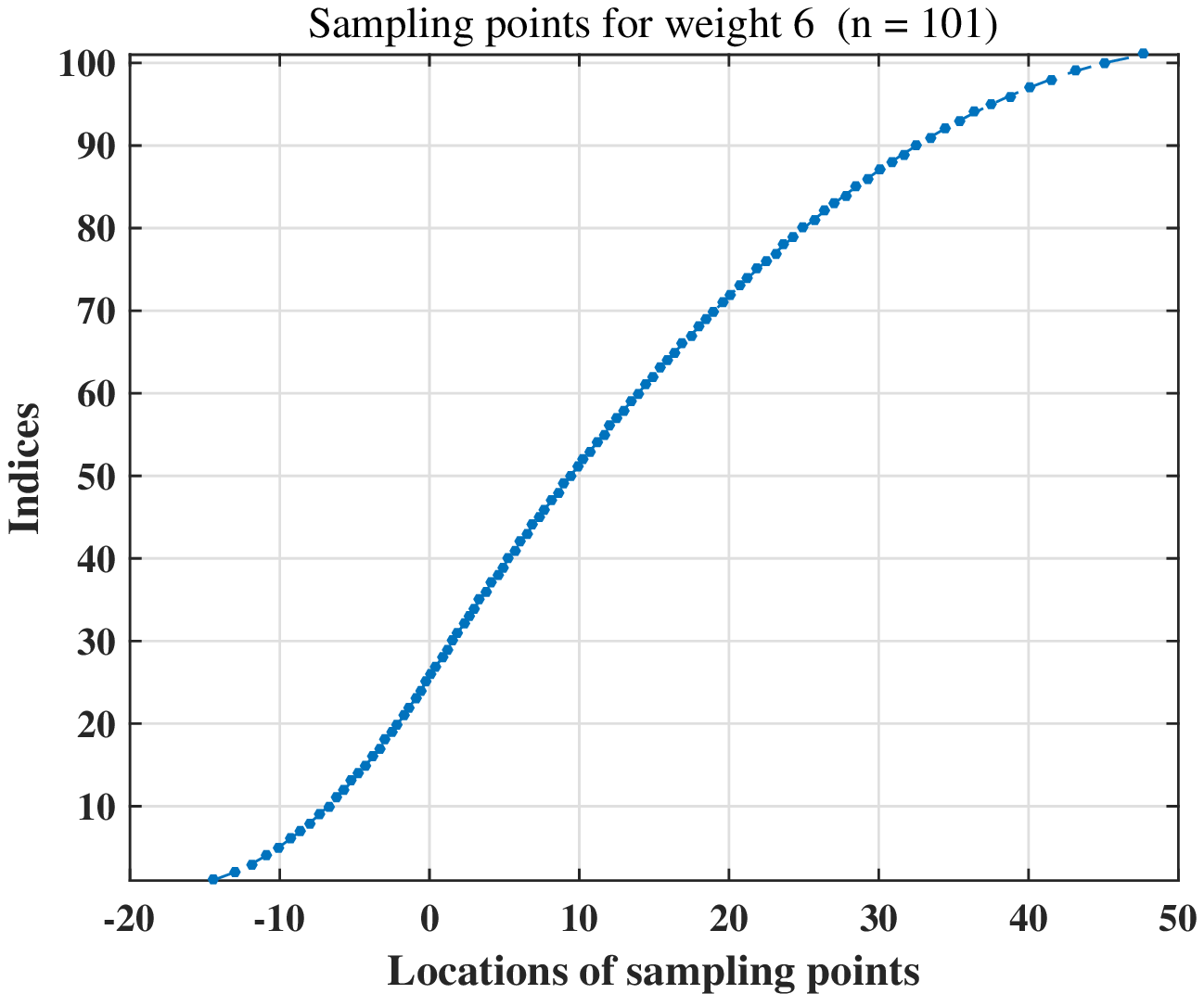}
\end{center}
\caption{Sampling points for $d_{6}$ and $w_{6}$ ($n = 101$). }
\label{fig:sample_w6_uneven_de}
\end{minipage}
\quad
\begin{minipage}[t]{0.48\linewidth}
\begin{center}
\includegraphics[width=\linewidth]{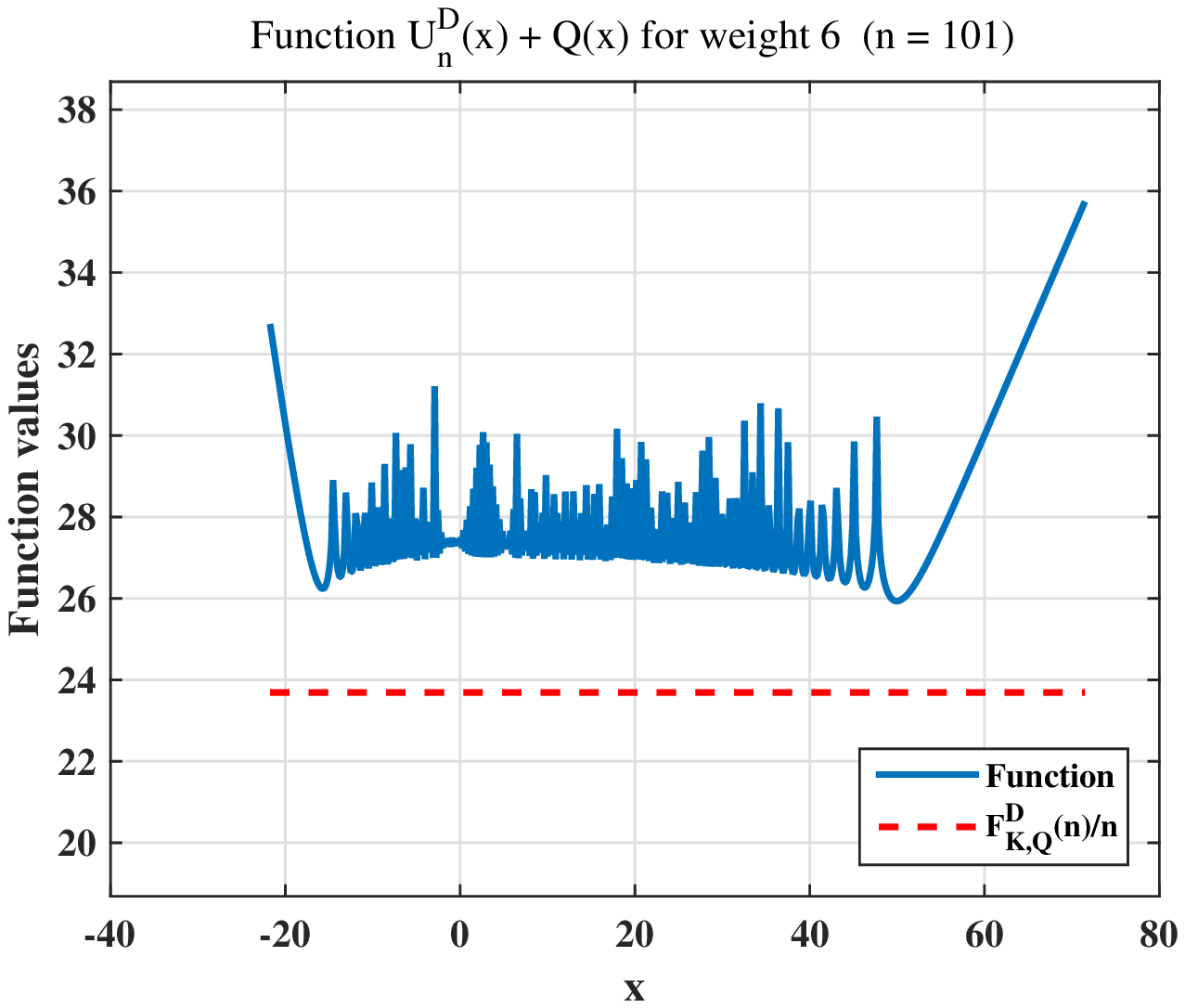}
\end{center}
\caption{Function $U_{n}^{\mathrm{D}}(\tilde{a}; x) + Q(x)$ for $d_{6}$ and $w_{6}$ ($n = 101$). }
\label{fig:pot_w6_uneven_de}
\end{minipage}
\end{figure}

\begin{figure}[ht]
\begin{minipage}[t]{0.48\linewidth}
\begin{center}
\includegraphics[width=\linewidth]{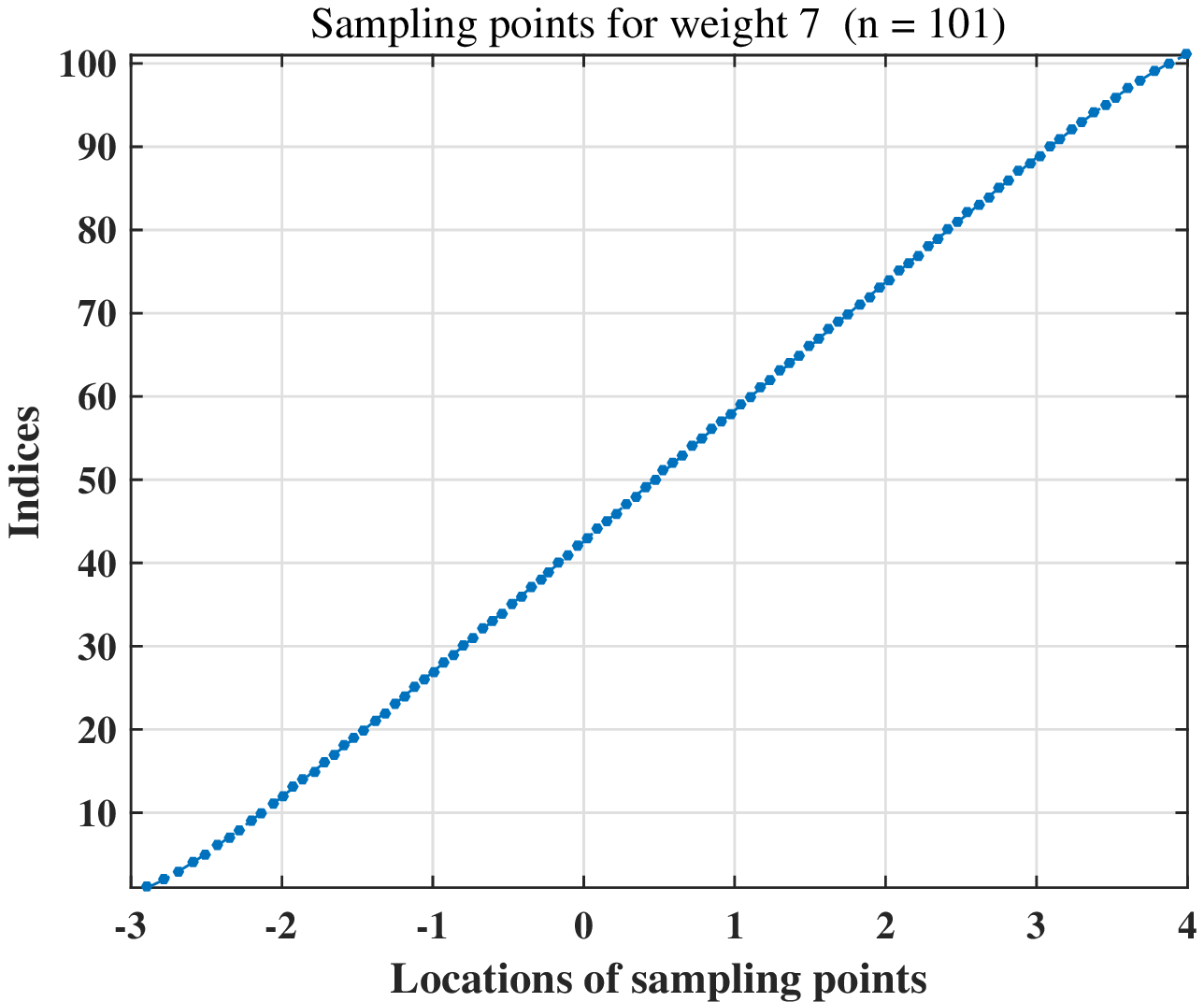}
\end{center}
\caption{Sampling points for $d_{7}$ and $w_{7}$ ($n = 101$). }
\label{fig:sample_w7_uneven_de}
\end{minipage}
\quad
\begin{minipage}[t]{0.48\linewidth}
\begin{center}
\includegraphics[width=\linewidth]{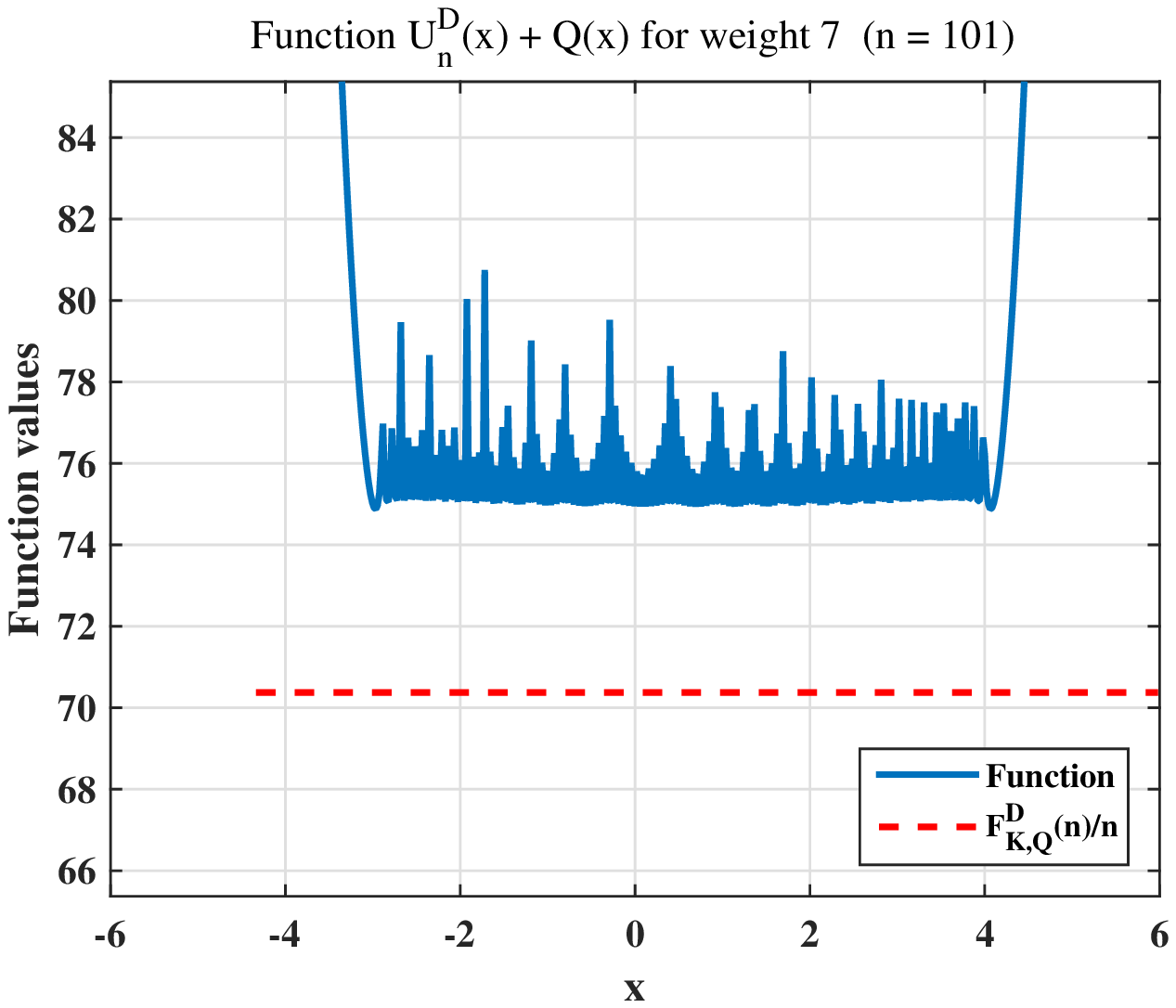}
\end{center}
\caption{Function $U_{n}^{\mathrm{D}}(\tilde{a}; x) + Q(x)$ for $d_{7}$ and $w_{7}$ ($n = 101$). }
\label{fig:pot_w7_uneven_de}
\end{minipage}
\end{figure}

\begin{figure}[ht]
\begin{minipage}[t]{0.48\linewidth}
\begin{center}
\includegraphics[width=\linewidth]{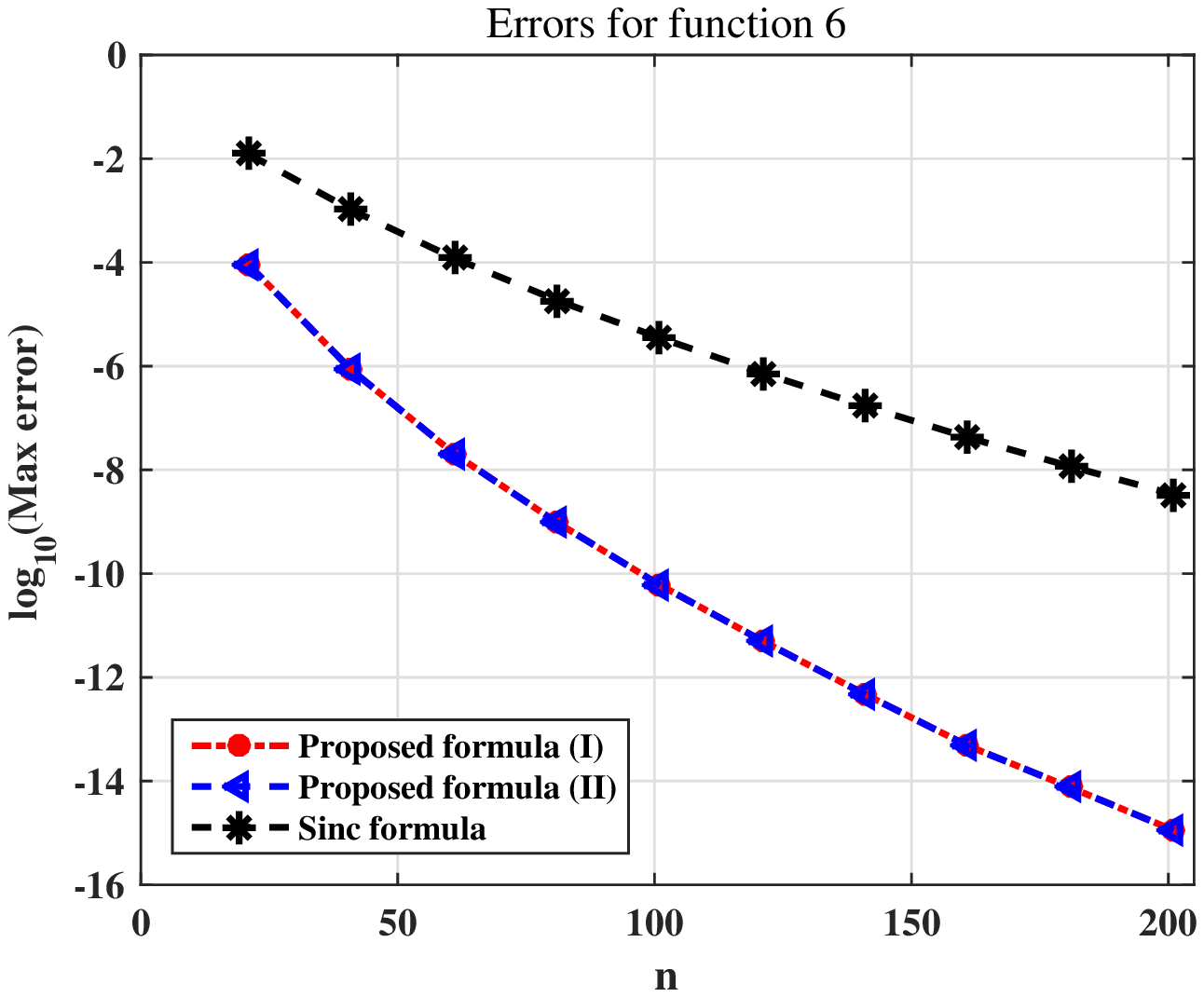}
\end{center}
\caption{Errors for function 6 ($f_{6}$). }
\label{fig:func_w6_uneven_se}
\end{minipage}
\quad
\begin{minipage}[t]{0.48\linewidth}
\begin{center}
\includegraphics[width=\linewidth]{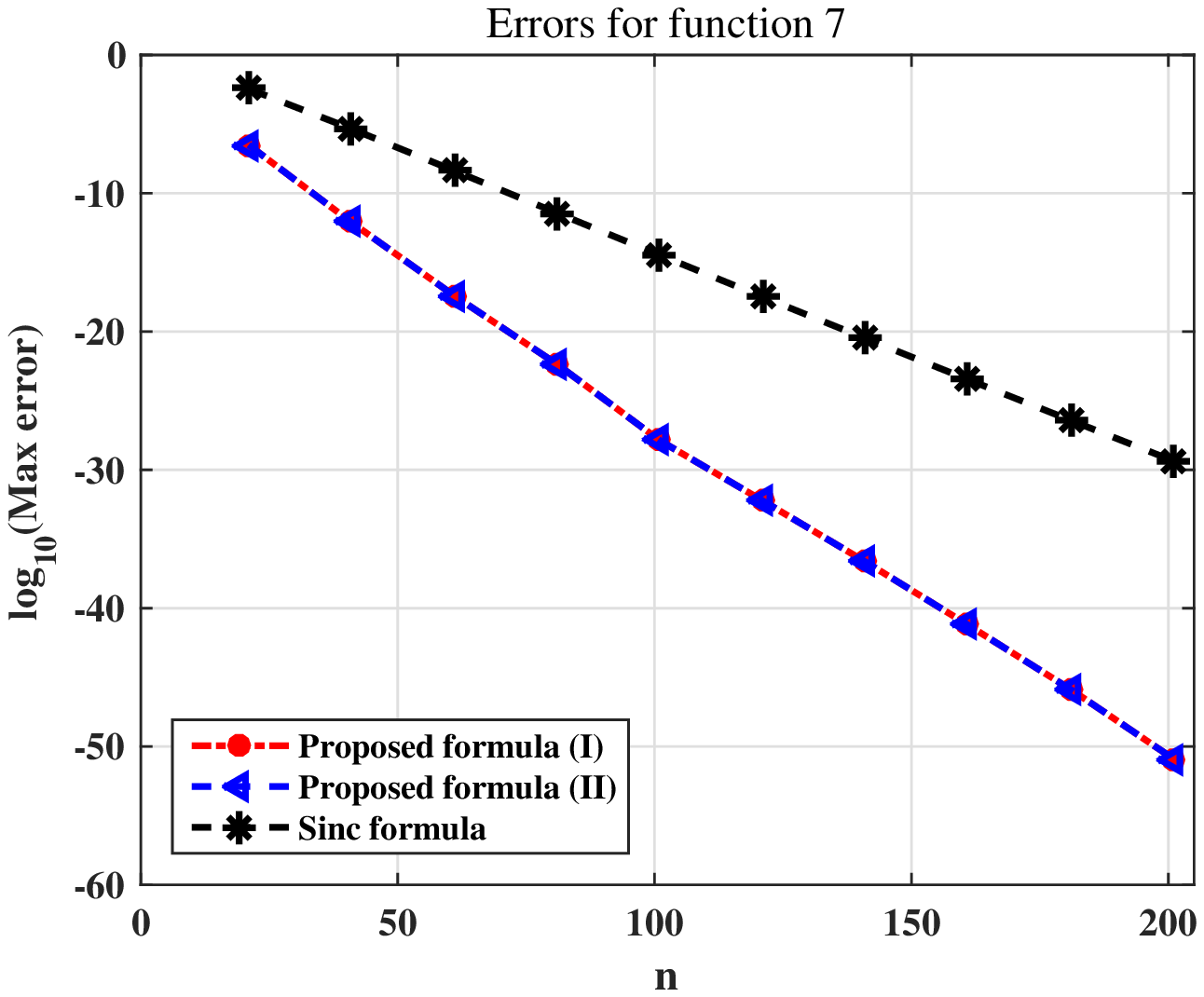}
\end{center}
\caption{Errors for function 7 ($f_{7}$). }
\label{fig:func_w7_uneven_de}
\end{minipage}
\end{figure}

\section{Concluding remarks}
\label{sec:concl}

In this paper, 
we proposed the method for designing accurate approximation formulas 
for functions in the spaces $\boldsymbol{H}^{\infty}(\mathcal{D}_{d}, w)$
by minimizing the discrete energy 
$I_{n}^{\mathrm{D}}$ in~\eqref{eq:def_D_ene} on $\mathcal{R}_{n} \subset \mathbf{R}^{n}$. 
On Assumptions~\ref{assump:w} and~\ref{assump:w_convex},
we proved that $I_{n}^{\mathrm{D}}$ is strictly convex on $\mathcal{R}_{n}$, 
and hence we showed that we can obtain the optimal solution $a^{\ast} \in \mathcal{R}_{n}$
approximately by the standard technique in convex optimization. 
Then, by using $a^{\ast}$ as a set of sampling points, 
we designed the approximation formula $L_{n}[a^{\ast};f]$ in~\eqref{eq:ProposedApproxFormula}
and gave the upper bound of its error for each space $\boldsymbol{H}^{\infty}(\mathcal{D}_{d}, w)$. 
By the numerical experiments, we showed the formula is accurate. 

Major themes for future work are
finding the precise orders of the errors of the proposed formulas with respect to $n$ 
and investigating whether the proposed formulas are asymptotically optimal or not. 
In addition, other themes will be 
their applications to various numerical methods such as the
numerical integration and solving differential/integral equations.  
In fact, in 
\ifthenelse{\value{journal} = 0}{%
the paper \cite{bib:TaOkaSu_PotNumInt_2017}, 
}{
\cite{bib:TaOkaSu_PotNumInt_2017}, 
}
we have considered the application of the previous formulas 
in \cite{bib:TaOkaSu_PotApprox_2017} to numerical integration.

\section*{Funding}

K. Tanaka is supported by the grant-in-aid of 
Japan Society of the Promotion of Science with 
KAKENHI Grant Number 17K14241. 

\section*{Acknowledgement}

The authors would like to give thanks to 
Dr.~Kuan Xu for his valuable comments about Remark~\ref{rem:ex_func_on_entire_R}. 

\ifthenelse{\value{journal} = 0}{%

}{
\input{biblio_IMA.tex}
}

\appendix

\section{Estimate of the difference $F_{K,Q}^{\mathrm{C}}(n) - F_{K,Q}^{\mathrm{D}}(n)$}
\label{sec:diff_FC_FD}

We provide an upper bound of the difference 
$F_{K,Q}^{\mathrm{C}}(n) - F_{K,Q}^{\mathrm{D}}(n)$. 
More precisely, 
on the assumption that
\begin{align}
\max_{1 \leq i \leq n-1} |a_{i+1}^{\ast} - a_{i}^{\ast} | \leq 1
\label{eq:assump_max_distance_a_ast}
\end{align}
for the minimizer $a^{\ast} \in \mathcal{R}_{n}$ of $I_{n}^{\mathrm{D}}$, 
we show that
\begin{align} 
F_{K,Q}^{\mathrm{C}}(n)
-
F_{K,Q}^{\mathrm{D}}(n)
\leq 
- (3n+1) \log h_{a^{\ast}} + C_{n} + e_{n}^{\ast}(Q), 
\label{eq:FCvsFD_step2_abstract}
\end{align}
where 
\begin{align}
h_{a^{\ast}} = \min_{1 \leq i \leq n-1} \left| a_{i+1}^{\ast} - a_{i}^{\ast} \right|
\label{eq:sep_dist_a_ast}
\end{align}
is the separation distance of $a^{\ast}$, 
the value $C_{n}$ is independent of $a^{\ast}$ with $C_{n} = \mathrm{O}(n) \ (n\to \infty)$, and 
$e_{n}^{\ast}(Q)$ is a sum of the differences of some integrations of $Q$ given by $a^{\ast}$. 
A concrete expression of the upper bound is given by Proposition~\ref{prop:diff_FC_FD} below. 
We prove it after showing several lemmas. 


The first lemma shows that $I_{n}^{\mathrm{C}}(\mu_{n}^{\ast})$ is monotonically increasing with respect to $n$. 

\begin{lem}
\label{lem:monotone_I}
For an integer $n \geq 1$, 
the inequality 
$I_{n}^{\mathrm{C}}(\mu_{n}^{\ast}) \leq I_{n+1}^{\mathrm{C}}(\mu_{n+1}^{\ast})$
holds true. 
\end{lem}

\begin{proof}
Let $\eta_{i}^{\ast}$ be the copy of $\mu_{n+1}^{\ast} / (n+1)$ for $i = 1,\ldots, n+1$. 
Then, for $k = 1,\ldots, n+1$, we have
\begin{align*}
I_{n+1}^{\mathrm{C}}(\mu_{n+1}^{\ast}) 
= & \, I_{n+1}^{\mathrm{C}}\left( \sum_{i=1}^{n+1} \eta_{i}^{\ast} \right) \\
= & \, 
\sum_{i=1}^{n+1} 
\sum_{j=1}^{n+1} 
\int_{\mathbf{R}} \int_{\mathbf{R}} K(x - y) \, \mathrm{d}\eta_{i}^{\ast}(x) \mathrm{d}\eta_{j}^{\ast}(y)
+ 2 \sum_{i=1}^{n+1} 
\int_{\mathbf{R}} Q(x) \, \mathrm{d}\eta_{i}^{\ast}(x) \\
\geq & \, 
\sum_{\begin{subarray}{c} i=1 \\ i \neq k \end{subarray}}^{n+1} 
\sum_{\begin{subarray}{c} j=1 \\ j \neq k \end{subarray}}^{n+1} 
\int_{\mathbf{R}} \int_{\mathbf{R}} K(x - y) \, \mathrm{d}\eta_{i}^{\ast}(x) \mathrm{d}\eta_{j}^{\ast}(y)
+ 2 \sum_{\begin{subarray}{c} i=1 \\ i \neq k \end{subarray}}^{n+1} 
\int_{\mathbf{R}} Q(x) \, \mathrm{d}\eta_{i}^{\ast}(x) \\
& + \sum_{i = 1}^{n+1} 
\int_{\mathbf{R}} \int_{\mathbf{R}} K(x - y) \, \mathrm{d}\eta_{i}^{\ast}(x) \mathrm{d}\eta_{k}^{\ast}(y) 
+ 2 \int_{\mathbf{R}} Q(x) \, \mathrm{d}\eta_{k}^{\ast}(x) \\
= & \, I_{n}^{\mathrm{C}}\left( \sum_{\begin{subarray}{c} i=1 \\ i \neq k \end{subarray}}^{n+1}  \eta_{i}^{\ast} \right)
+ \sum_{i = 1}^{n+1} 
\int_{\mathbf{R}} \int_{\mathbf{R}} K(x - y) \, \mathrm{d}\eta_{i}^{\ast}(x) \mathrm{d}\eta_{k}^{\ast}(y) 
+ 2 \int_{\mathbf{R}} Q(x) \, \mathrm{d}\eta_{k}^{\ast}(x) \\
\geq & \, 
I_{n}^{\mathrm{C}}(\mu_{n}^{\ast})
+ \sum_{i = 1}^{n+1} 
\int_{\mathbf{R}} \int_{\mathbf{R}} K(x - y) \, \mathrm{d}\eta_{i}^{\ast}(x) \mathrm{d}\eta_{k}^{\ast}(y) 
+ 2 \int_{\mathbf{R}} Q(x) \, \mathrm{d}\eta_{k}^{\ast}(x). 
\end{align*}
Then, by summing up both sides of this inequality for $k = 1,\ldots, n+1$, we have
\begin{align*}
(n+1) \, I_{n+1}^{\mathrm{C}}(\mu_{n+1}^{\ast}) 
\geq 
(n+1) \, I_{n}^{\mathrm{C}}(\mu_{n}^{\ast}) + I_{n+1}^{\mathrm{C}}(\mu_{n+1}^{\ast})
\iff
\frac{I_{n+1}^{\mathrm{C}}(\mu_{n+1}^{\ast})}{n+1} 
\geq 
\frac{I_{n}^{\mathrm{C}}(\mu_{n}^{\ast})}{n}. 
\end{align*}
Hence $I_{n}^{\mathrm{C}}(\mu_{n}^{\ast})/n$ is monotonically increasing, and so is $I_{n}^{\mathrm{C}}(\mu_{n}^{\ast})$. 
\end{proof}


As an approximation of the measure 
$\sum_{i=1}^{n} \delta_{a_{i}} \in \mathcal{M}_{\mathrm{c}}(\mathbf{R}, n)$ 
for $a = (a_{1}, \ldots , a_{n}) \in \mathcal{R}_{n}$, 
we consider the Borel measure $\nu_{\hat{a}} \in \mathcal{M}_{\mathrm{c}}(\mathbf{R}, n+1)$ defined by
\[
\nu_{\hat{a}}(Z) = \sum_{i=0}^{n} \frac{1}{a_{i+1} - a_{i}} \int_{Z \cap [a_{i}, a_{i+1}]} \mathrm{d}y
\]
for $\hat{a} = (a_{0}, a_{1}, \ldots , a_{n}, a_{n+1}) \in \mathcal{R}_{n+2}$. 
Then, we define 
\begin{align}
S_{i}(\hat{a}) 
& = \int_{a_{i-1}}^{a_{i+1}} K(a_{i} - y) \, \mathrm{d}\nu_{\hat{a}}(y), \label{eq:def_S_int} \\
T_{i}(\hat{a})
&= \int_{a_{i}}^{a_{i+1}} \int_{a_{i}}^{a_{i+1}} K(x - y) \, \mathrm{d}\nu_{\hat{a}}(y) \mathrm{d}\nu_{\hat{a}}(x), \label{eq:def_T_intint} \\
e_{n}^{(1)}(Q; \hat{a})
& =
\int_{a_{0}}^{a_{n+1}} Q(x) \, \mathrm{d}\nu_{\hat{a}}(x)
- 
\frac{n-1}{n} \sum_{i=1}^{n} Q(a_{i}), \\
e_{n}^{(2)}(Q; \hat{a})
& =
\int_{a_{0}}^{a_{n+1}} Q(x) \, \mathrm{d}\nu_{\hat{a}}(x)
-
\int_{\mathbf{R}} Q(x) \, \mathrm{d}\mu_{n}^{\ast}(x)
\end{align}
for $\hat{a}$. 
By using these expressions, we can give a preliminary upper bound of 
$F_{K,Q}^{\mathrm{C}}(n) - F_{K,Q}^{\mathrm{D}}(n)$ as shown in the following lemma. 

\begin{lem}
\label{lem:FCvsFD_step1}
Let $a^{\ast} = (a_{1}^{\ast}, \ldots , a_{n}^{\ast}) \in \mathcal{R}_{n}$ be the minimizer of $I_{n}^{\mathrm{D}}$, 
and choose $a_{0}^{\ast}$ and $a_{n+1}^{\ast}$ such that 
$\hat{a}^{\ast} = (a_{0}^{\ast}, a_{1}^{\ast}, \ldots , a_{n}^{\ast}, a_{n+1}^{\ast}) \in \mathcal{R}_{n+2}$.  
Then, we have
\begin{align}
F_{K,Q}^{\mathrm{C}}(n)
-
F_{K,Q}^{\mathrm{D}}(n)
\leq 
\sum_{i=1}^{n} S_{i}(\hat{a}^{\ast})
+ 
\sum_{i=0}^{n} T_{i}(\hat{a}^{\ast})
+ 
e_{n}^{(1)}(Q; \hat{a}^{\ast})
+ 
e_{n}^{(2)}(Q; \hat{a}^{\ast}). 
\label{eq:FCvsFD_step1}
\end{align}
\end{lem}

\begin{proof}
We consider a general element $\hat{a} = (a_{0}, a_{1}, \ldots, a_{n}, a_{n+1}) \in \mathcal{R}_{n+2}$. 
By noting the convexity of $K$, for $x \in (a_{i}, a_{i+1})$ we have
\begin{align}
& \left( \int_{a_{0}}^{a_{n+1}} - \int_{a_{i}}^{a_{i+1}} \right) K(x - y) \, \mathrm{d}\nu_{\hat{a}}(y)
\leq
\sum_{j=1}^{n} K(x - a_{j}) \notag \\
& \iff
\int_{a_{0}}^{a_{n+1}} K(x - y) \, \mathrm{d}\nu_{\hat{a}}(y)
\leq
\sum_{j=1}^{n} K(x - a_{j})
+ 
\int_{a_{i}}^{a_{i+1}} K(x - y) \, \mathrm{d}\nu_{\hat{a}}(y). 
\label{eq:Key_ineq_K_1}
\end{align}
In a similar manner, we have
\begin{align}
\int_{a_{0}}^{a_{n+1}} K(a_{i} - y) \, \mathrm{d}\nu_{\hat{a}}(y)
\leq
\sum_{\begin{subarray}{c} j=1 \\ j\neq i \end{subarray}}^{n} K(a_{i} - a_{j})
+ 
\int_{a_{i-1}}^{a_{i+1}} K(a_{i} - y) \, \mathrm{d}\nu_{\hat{a}}(y). 
\label{eq:Key_ineq_K_2}
\end{align}
Then, 
by using Lemma~\ref{lem:monotone_I}, 
the fact that $\nu_{\hat{a}} \in \mathcal{M}_{\mathrm{c}}(\mathbf{R}, n+1)$,
and Inequality~\eqref{eq:Key_ineq_K_1}, 
we can bound the optimal value $I_{n}^{\mathrm{C}}(\mu_{n}^{\ast})$ from above as follows: 
\begin{align*}
I_{n}^{\mathrm{C}}(\mu_{n}^{\ast})
\leq 
I_{n+1}^{\mathrm{C}}(\mu_{n+1}^{\ast})
& \leq 
I_{n+1}^{\mathrm{C}}(\nu_{\hat{a}}) \\
& = 
\int_{a_{0}}^{a_{n+1}} \int_{a_{0}}^{a_{n+1}} K(x - y) \, \mathrm{d}\nu_{\hat{a}}(y) \mathrm{d}\nu_{\hat{a}}(x) 
+ 2 \int_{a_{0}}^{a_{n+1}} Q(x) \, \mathrm{d}\nu_{\hat{a}}(x) \\
& \leq 
\sum_{i=1}^{n} R_{i}(\hat{a})
+ 
\sum_{i=0}^{n} T_{i}(\hat{a})
+ 
2 \int_{a_{0}}^{a_{n+1}} Q(x) \, \mathrm{d}\nu_{\hat{a}}(x), 
\end{align*}
where 
\begin{align}
R_{i}(\hat{a}) = \int_{a_{0}}^{a_{n+1}} K(x - a_{i}) \, \mathrm{d}\nu_{\hat{a}}(x). \notag 
\end{align}
Furthermore, by using Inequality~\eqref{eq:Key_ineq_K_2}, we have
\begin{align*}
R_{i}(\hat{a})
\leq
\sum_{\begin{subarray}{c} j=1 \\ j\neq i \end{subarray}}^{n} K(a_{i} - a_{j})
+ 
\int_{a_{i-1}}^{a_{i+1}} K(a_{i} - y) \, \mathrm{d}\nu_{a}(y). 
\end{align*}
Therefore, for $a = (a_{1}, \ldots, a_{n}) \in \mathcal{R}_{n}$ we have
\begin{align}
I_{n}^{\mathrm{C}}(\mu_{n}^{\ast})
& \leq  
\sum_{i=1}^{n} 
\sum_{\begin{subarray}{c} j=1 \\ j\neq i \end{subarray}}^{n} K(a_{i} - a_{j})
+
\sum_{i=1}^{n} S_{i}(\hat{a})
+ 
\sum_{i=0}^{n} T_{i}(\hat{a})
+ 
2 \int_{a_{0}}^{a_{n+1}} Q(x) \, \mathrm{d}\nu_{\hat{a}}(x) \notag \\
& \leq 
I_{n}^{\mathrm{D}}(a)
+
\sum_{i=1}^{n} S_{i}(\hat{a})
+ 
\sum_{i=0}^{n} T_{i}(\hat{a})
+ 
2 e_{n}^{(1)}(Q; \hat{a}). \notag
\end{align}
Finally, by letting $a = a^{\ast}$ and choosing $a_{0}^{\ast}$ and $a_{n+1}^{\ast}$ 
such that $\hat{a}^{\ast} = (a_{0}^{\ast}, a_{1}^{\ast}, \ldots , a_{n}^{\ast}, a_{n+1}^{\ast}) \in \mathcal{R}_{n+2}$, 
we have 
\begin{align}
F_{K, Q}^{\mathrm{C}}(n)
& \leq 
I_{n}^{\mathrm{D}}(a^{\ast})
+
\sum_{i=1}^{n} S_{i}(\hat{a}^{\ast})
+ 
\sum_{i=0}^{n} T_{i}(\hat{a}^{\ast})
+ 
2 e_{n}^{(1)}(Q; \hat{a}^{\ast})
- \int_{\mathbf{R}} Q(x) \, \mathrm{d}\mu_{n}^{\ast}(x) \notag \\
& = 
F_{K, Q}^{\mathrm{D}}(n)
+
\sum_{i=1}^{n} S_{i}(\hat{a}^{\ast})
+ 
\sum_{i=0}^{n} T_{i}(\hat{a}^{\ast})
+ 
e_{n}^{(1)}(Q; \hat{a}^{\ast})
+ 
e_{n}^{(2)}(Q; \hat{a}^{\ast}), \notag
\end{align}
which is Inequality~\eqref{eq:FCvsFD_step1}.
\end{proof}


In order to bound $S_{i}(\hat{a}^{\ast})$ and $T_{i}(\hat{a}^{\ast})$ from above, 
we use the fact that the function $K$ given by~\eqref{eq:SettingK} satisfies
\begin{align}
K(x) 
\leq 
- \log \left| \left( \tanh \frac{\pi}{4d} \right) x \right| 
\leq 
- \log |x| + c_{d}
\quad \text{for $x$ with $|x| \leq 1$}, 
\label{eq:K_less_than_mlog}
\end{align}
where $c_{d} \geq 0$ is given by
\begin{align}
c_{d} = - \log \left( \tanh \frac{\pi}{4d} \right). 
\label{eq:def_c_d}
\end{align}

\begin{lem}
Let $a^{\ast} = (a_{1}^{\ast}, \ldots , a_{n}^{\ast}) \in \mathcal{R}_{n}$ be the minimizer of $I_{n}^{\mathrm{D}}$
satisfying \eqref{eq:assump_max_distance_a_ast}, 
and let $a_{0}^{\ast}$ and $a_{n+1}^{\ast}$ be chosen such that 
$\hat{a}^{\ast} = (a_{0}^{\ast}, a_{1}^{\ast}, \ldots , a_{n}^{\ast}, a_{n+1}^{\ast}) \in \mathcal{R}_{n+2}$
and $h_{a^{\ast}} \leq | a_{j+1}^{\ast} - a_{j}^{\ast} | \leq 1$ for $j = 0, n$, 
where $h_{a^{\ast}}$ is the separation distance given by \eqref{eq:sep_dist_a_ast}. 
Then, for $i = 1, \ldots, n$, 
we have
\begin{align}
S_{i}(\hat{a}^{\ast}) 
& \leq -2 \log h_{a^{\ast}} + 2(1+c_{d}), 
\label{eq:S_i_ast} \\
T_{i}(\hat{a}^{\ast})   
& \leq - \log h_{a^{\ast}} + \frac{1}{2} + c_{d} 
\label{eq:T_i_ast}. 
\end{align}
\end{lem}

\begin{proof}
We begin with $S_{i}(\hat{a}^{\ast})$. 
By using~\eqref{eq:K_less_than_mlog}, we have
\begin{align*}
S_{i}(\hat{a}^{\ast})  
& \leq 
-\frac{1}{a_{i}^{\ast} - a_{i-1}^{\ast}} \int_{a_{i-1}^{\ast}}^{a_{i}^{\ast}} \log |a_{i}^{\ast} - y| \, dy
-\frac{1}{a_{i+1}^{\ast} - a_{i}^{\ast}} \int_{a_{i}^{\ast}}^{a_{i+1}^{\ast}} \log |a_{i}^{\ast} - y| \, dy
+2 c_{d} \\
& = - \log |a_{i+1}^{\ast} - a_{i}^{\ast}| - \log |a_{i}^{\ast} - a_{i-1}^{\ast}| + 2(1+c_{d}) \\
& \leq -2 \log h_{a^{\ast}} + 2(1+c_{d}). 
\end{align*}
Next, we bound $T_{i}(\hat{a}^{\ast})$ from above. 
For the inner integral in~\eqref{eq:def_T_intint}, we use~\eqref{eq:K_less_than_mlog} to obtain 
\begin{align*}
\int_{a_{i}^{\ast}}^{a_{i+1}^{\ast}} K(x - y) \, \mathrm{d}\nu_{a^{\ast}}(y)
\leq 
-\frac{1}{a_{i+1}^{\ast} - a_{i}^{\ast}}
\left[
(a_{i+1}^{\ast} - x) \log (a_{i+1}^{\ast} - x) + 
(x - a_{i}^{\ast}) \log (x - a_{i}^{\ast})
\right]
+ c_{d}.
\end{align*}
Therefore, we have
\begin{align*}
T_{i}(\hat{a}^{\ast}) 
& \leq
-\frac{1}{(a_{i+1}^{\ast} - a_{i}^{\ast})^{2}}
\left[
(a_{i+1}^{\ast} - a_{i}^{\ast})^{2} \log (a_{i+1}^{\ast} - a_{i}^{\ast})  
- \frac{1}{2} (a_{i+1}^{\ast} - a_{i}^{\ast})^{2}
\right]
+ c_{d} \\
& = - \log (a_{i+1}^{\ast} - a_{i}^{\ast}) + \frac{1}{2} + c_{d} \\
& \leq - \log h_{a^{\ast}} + \frac{1}{2} + c_{d}. 
\end{align*}
\end{proof}


Here we are in a position to provide an upper bound of $F_{K,Q}^{\mathrm{C}}(n)- F_{K,Q}^{\mathrm{D}}(n)$. 

\begin{prop}
\label{prop:diff_FC_FD}
Let $a^{\ast} = (a_{1}^{\ast}, \ldots , a_{n}^{\ast}) \in \mathcal{R}_{n}$ be the minimizer of $I_{n}^{\mathrm{D}}$
satisfying \eqref{eq:assump_max_distance_a_ast}, 
and let $a_{0}^{\ast}$ and $a_{n+1}^{\ast}$ be chosen such that 
$\hat{a}^{\ast} = (a_{0}^{\ast}, a_{1}^{\ast}, \ldots , a_{n}^{\ast}, a_{n+1}^{\ast}) \in \mathcal{R}_{n+2}$
and $h_{a^{\ast}} \leq | a_{j+1}^{\ast} - a_{j}^{\ast} | \leq 1$ for $j = 0, n$, 
where $h_{a^{\ast}}$ is the separation distance given by \eqref{eq:sep_dist_a_ast}. 
Then, we have
\begin{align}
F_{K,Q}^{\mathrm{C}}(n)
- 
F_{K,Q}^{\mathrm{D}}(n)
\leq 
-
(3n+1) \log h_{a^{\ast}}
+ 
C_{n}
+ 
e_{n}^{(1)}(Q; \hat{a}^{\ast})
+ 
e_{n}^{(2)}(Q; \hat{a}^{\ast}), 
\label{eq:FCvsFD_step2}
\end{align}
where
\begin{align*}
C_{n} = \left( \frac{5}{2} + 3c_{d} \right) n + \frac{1}{2} + c_{d}.
\end{align*}
\end{prop}

\begin{proof}
By combining Inequalities~\eqref{eq:FCvsFD_step1}, \eqref{eq:S_i_ast} and~\eqref{eq:T_i_ast}, 
we have~\eqref{eq:FCvsFD_step2}. 
\end{proof}


In addition, 
as a corollary of Proposition~\ref{prop:diff_FC_FD}, 
we can provide another error estimate of the proposed formula in terms of $F_{K,Q}^{\mathrm{C}}(n)$. 

\begin{cor}
On the same assumption as Proposition~\ref{prop:diff_FC_FD}, 
we have
\begin{align}
\sup_{
\begin{subarray}{c}
f \in \boldsymbol{H}^{\infty}(\mathcal{D}_{d}, w) \\
\| f \| \leq 1
\end{subarray}}
\left(
\sup_{x \in \mathbf{R}}
\left|
f(x) - L_{n}[a^{\ast}; f](x)
\right|
\right)
\leq
\frac{\hat{C}_{n} \, \hat{D}_{n}}{(h_{a^{\ast}})^{3+1/n}} \exp \left(- \frac{F_{K,Q}^{\mathrm{C}}(n)}{n} \right), 
\label{eq:ErrorEstimate_FC}
\end{align}
where 
$\hat{C}_{n} = \exp (C_{n}/n)$ and 
$\hat{D}_{n} = \exp \left[ \left( e_{n}^{(1)}(Q; \hat{a}^{\ast}) + e_{n}^{(2)}(Q; \hat{a}^{\ast}) \right)/n \right]$. 
\end{cor}

\begin{proof}
The conclusion follows from Inequalities~\eqref{eq:ErrorEstimate} and~\eqref{eq:FCvsFD_step2}.
\end{proof}


\begin{rem}
We have not succeeded in estimating the separation distance $h_{a^{\ast}}$ yet. 
However, from numerical experiments, we observed that 
\begin{align}
h_{a^{\ast}} 
\sim 
\begin{cases}
n^{-1/2} & (Q(x) \approx \beta |x|), \\
(\log n)/n & (Q(x) \approx \beta \exp(|x|)).
\end{cases}
\label{eq:sep_dist_h}
\end{align}
Furthermore, in 
\ifthenelse{\value{journal} = 0}{%
\cite[Section 5.2]{bib:TaOkaSu_PotApprox_2017}%
}{
\citet[Section 5.2]{bib:TaOkaSu_PotApprox_2017}%
}
we have had the rough estimates:
\begin{align}
\frac{F_{K,Q}^{\mathrm{C}}(n)}{n}
\sim
\begin{cases}
n^{1/2} & (Q(x) \approx \beta |x|), \\
n/\log n & (Q(x) \approx \beta \exp(|x|)).
\end{cases}
\label{eq:FC_over_n}
\end{align}
Therefore, we expect that 
\begin{align*}
\frac{\hat{C}_{n}}{(h_{a^{\ast}})^{3+1/n}} \exp \left(- \frac{F_{K,Q}^{\mathrm{C}}(n)}{n} \right) 
\sim
\begin{cases}
n^{(3+1/n)/2} \, \exp(-c' n^{1/2}) & (Q(x) \approx \beta |x|), \\
(n/\log n)^{3+1/n} \, \exp(-c'' n/\log n) & (Q(x) \approx \beta \exp(|x|)). 
\end{cases}
\end{align*}
However, 
precise estimates such as \eqref{eq:sep_dist_h} and~\eqref{eq:FC_over_n} may be difficult. 
Therefore, 
these estimates, as well as the estimate of $\hat{D}_{n}$, 
are themes of our future work. 
\end{rem}

\end{document}